\documentclass[12pt]{article}
\usepackage{amssymb}
\usepackage{amsfonts}
\usepackage{amsmath}
\usepackage{cases}
\usepackage[usenames]{color}
\usepackage{mathrsfs}
\usepackage{CJK}
\usepackage{cite}
\usepackage{cases}
\usepackage{amsthm}
\usepackage{indentfirst}
\usepackage{setspace}
\usepackage{abstract}
\usepackage[none]{hyphenat}

\allowdisplaybreaks[4]
\def\sc{\scriptstyle}

\def\cl{\centerline}

\def\al{\alpha}

\def\b{\beta}

\def\vs{\vspace*}
\def\W{\mathcal{W}}

\def\R{\mathcal{R}}
\def\H{\mathcal{H}}
\def\L{\mathcal{L}}

\def\Z{\mathbb{Z}}
\def\N{\mathbb{N}}

\def\C{\mathbb{C}}

\def\C{\mathbb{C}}

\numberwithin{equation}{section}
\newtheorem{theo}{Theorem}[section]
\newtheorem{defi}[theo]{Definition}

\newtheorem{lemm}[theo]{Lemma}

\newtheorem{prop}[theo]{Proposition}

\newtheorem{remark}[theo]{Remark}

\begin{document}
\sloppy{}
\begin{CJK*}{GBK}{song}
\baselineskip 6pt
\lineskip 6pt

\begin{center}{\Large\bf Simple  non-weight modules over \\ Lie superalgebras of Block type}
\footnote {Supported by NSF Grant No.~11971350 of China.}
\end{center}
\vs{6pt}

\cl{Yucai Su,\ \ Xiaoqing Yue,  \ \ Xiaoyu Zhu}
\cl{\footnotesize School of Mathematical Sciences, Tongji University, Shanghai 200092, China}
\cl{\footnotesize E-mails: ycsu@tongji.edu.cn, xiaoqingyue@tongji.edu.cn, 1810079@tongji.edu.cn}
\vs{12pt} \par

\noindent{{\bf Abstract.} In this paper, a family of non-weight modules over Lie superalgebras $S(q)$ of Block type are studied. 
 Free $U(\eta)$-modules of rank $1$ over  Ramond-Block algebras and free $U(\mathfrak{h})$-modules of rank $2$ over  Neveu-Schwarz-Block algebras
 are constructed and classified. Moreover, the sufficient and necessary conditions for such modules to be simple are presented, and their isomorphism classes are also determined. The results cover some existing results.}
\vs{6pt}



\section{Introduction}
 Lie algebras of Block type (cf. \cite{B}) were firstly introduced by Block in 1958. Since these Lie algebras have close relations with some well-known Lie algebras,
such as the Virasoro algebra (cf. \cite{KPS}), $W$-infinity algebras (cf. \cite{B1}), ect., the structure theory of  Lie (super)algebras of this type have been
intensively studied by many authors (cf. \cite{DZ,OZ,X1,Z}). 
In addition, Lie algebras of Block type are also special cases of Lie algebras of (generalized) Cartan type (cf. \cite{X}).
It is well known that although Lie algebras of Cartan type have a long history, their representation theory is however far from being well developed.
It is very helpful to first study the representation theory of special cases 
in order to better understand the representation theory of Lie algebras of Cartan type. Besides, we also notice that Lie algebras of Block type have some applications in the
integrable system (cf.~\cite{18,27}). Partially due to this, the study of representations
 of Lie (super)algebras of this kind becomes significant and has been attracting a lot of attentions. 

 Non-weight modules constitute one of the important ingredients of the representation theory of Lie (super)algebras, 
 a classification of non-weight modules over them under some conditions  is definitely necessary and is interesting as well.
 Recently, some authors constructed and studied some important classes of non-weight modules on which the Cartan subalgebra acts freely. These modules are called free $U(\mathfrak{h})$-modules, where $U(\mathfrak{h})$ is the universal enveloping algebra of the Cartan subalgebra $\mathfrak{h}$.
In \cite{N}, Nilsson first introduced them for the simple Lie algebra $sl_{n+1}$. Then in a subsequent paper \cite{N1}, Nilsson presented that free $U(\mathfrak{h})$-modules of a finite dimensional simple Lie algebra can exist only when it is of type $A$ or $C$. At the same time, these modules over $sl_{n+1}$ were showed by using a different way in \cite{TZ1}. The authors in \cite{TZ} proved that any free $U(\mathfrak{h})$-module of rank 1 over Witt algebra is isomorphic to $\Omega(\lambda,a)$ (cf. \cite{ LZ}) for some $\lambda\in \C^{*}$ and $a\in \C$. After that, with aid of the approach in \cite{N} and the results in \cite{TZ}, many authors have studied such simple modules over finite dimensional simple Lie algebras and some infinite dimensional Lie algebras (e.g., [\ref{CG}--\ref{CY}, \ref{HCS}, \ref{LG}, \ref{N1}, \ref{WZ}]). 


However, due to the complexity of Lie superalgebras themselves (especially  infinite dimensional Lie superalgebras), the problem of classifying non-weight modules over Lie superalgebras has so far received insufficient 
attentions in the literature (cf. \cite{YYX, YYX1}). We believe that a classification of non-weight modules of Lie superalgebras of Block type will certainly help us better understand the representation theory of Lie superalgebras of this type and other Lie superalgebras which are closely related to them, such as Lie superalgebra of Cartan type (cf. \cite{ZZ}), (generalized) super-Virasoro algebras (cf. \cite{SZ}). This is also our motivation to present the results here. In this paper, for a nonzero complex number $q$, we construct and classify the free $U(\C L_{0,0}\oplus\C G_{0,0})$-modules of rank 1 and the free $U(\C L_{0,0})$-modules of rank 2 respectively over the Lie superalgebra $S(q)$ of Block type defined in the following. 

\begin{defi}\label{def1}
 {\rm
\begin{itemize}
\item[(1)]
  Let $s=0$ or $\frac{1}{2}$ and $q\in \C^{*}$. {\it The Lie superalgebra $S(q)$ of Block type} (without center), is an infinite dimensional Lie superalgebra with the even part $S(q)_{\overline{0}}={\rm span}\{L_{m,i}\,|\,m\in \Z,\,i\in \Z_{+}\}$ and the odd part $S(q)_{\overline{1}}={\rm span}\{G_{l,j}\,|\,l\in s+\Z,\,j\in \Z_{+}\}$ together with the following relations
\begin{eqnarray}&\!\!\!\!\!\!\!\!\!\!\!\!\!\!\!\!\!\!\!\!\!\!\!\!&\label{J_a_b_}
[L_{m,i},L_{n,j}] =\big(n(i+q)-m(j+q)\big)L_{m+n,i+j},\\
&\!\!\!\!\!\!\!\!\!\!\!\!\!\!\!\!\!\!\!\!\!\!\!\!&\label{J_a_b c_}
[L_{m,i},G_{l,j}] =\Big(l\big(i+q\big)-m\Big(j+\frac{q}{2}\Big)\Big)G_{m+l,i+j},\\
&\!\!\!\!\!\!\!\!\!\!\!\!\!\!\!\!\!\!\!\!\!\!\!\!&\label{J_a_b_c d}
[G_{l,i},G_{r,j}] = 2qL_{l+r,i+j}.
\end{eqnarray}
\item[(2)]
The first class of Lie superalgebras  $S(q)$ with $s = 0$ are referred to as {\it Ramond-Block (RB) algebras}, denoted by $\mathcal{R}$, and the second class $S(q)$ with $s=\frac{1}{2}$ as {\it Neveu-Schwarz-Block (NSB) algebras}, denoted by $\mathcal{L}$.
\end{itemize}
}\end{defi}
Note that $S(q)$ contains a centerless super-Virasoro subalgebra $SVir[s]$, which is given by
\begin{equation*}
  SVir[s]={\rm span}\{L_{m},G_{l}\,|\, m\in \Z,\ l\in s+\Z\},
\end{equation*}
where $L_{m}=\frac{1}{q}L_{m,0}$ and $G_{l}=\frac{1}{q}G_{l,0}$. In particular, $SVir[0]$ is the centerless Ramond algebra and $SVir[\frac{1}{2}]$ is the centerless Neveu-Schwarz algebra.

Recall that the Witt algebra, also known as  the well-known (centerless) Virasoro algebra, $\W$ is the derivation Lie algebra of the Laurent polynomial algebra in one variable. More precisely, $\W=\bigoplus_{\al\in\Z}\C L_{\al}$ as a vector space and its Lie bracket is given by
\begin{equation}\label{Witt-1}
  [L_{\al},L_{\b}]=(\b-\al)L_{\al+\b},\ \ \ \forall \ \al,\,\b\in \Z.
\end{equation}
The Heisenberg-Virasoro algebra $\mathcal{H}$ is
the complex Lie algebra that has a basis $\{L_{\al},I_{\b},C_{1},C_{2},C_{3}\,|\,\al,\,\b\in \Z\}$ subject to the following Lie brackets:
\begin{eqnarray}&\!\!\!\!\!\!\!\!\!\!\!\!\!\!\!\!\!\!\!\!\!\!\!\!&
\label{Hei-1}  [L_{\al},L_{\b}] = (\b-\al)L_{\al+\b}+\delta_{\al+\b,0}\frac{\al^{3}-\al}{12}C_{1}, \\
  &\!\!\!\!\!\!\!\!\!\!\!\!\!\!\!\!\!\!\!\!\!\!\!\!&
  [L_{\al},I_{\b}] = \b I_{\al+\b}+\delta_{\al+\b,0}(\al^{2}+\al)C_{2}, \\
  &\!\!\!\!\!\!\!\!\!\!\!\!\!\!\!\!\!\!\!\!\!\!\!\!&
\label{Hei-3}  [I_{\al},I_{\b}] = \al \delta_{\al+\b,0}C_{3}.
\end{eqnarray}

\begin{remark}
\rm\begin{itemize}\parskip-3pt
              \item [(1)]The even part $SVir[s]_{\overline{0}}$ of the super-Virasoro algebra $SVir[s]$ is isomorphic to the Witt algebra $\W$.
              \item[(2)]Let $L$ be the subalgebra of $S(-1)$ generated by $\{L_{m,0}, L_{m,1}\,|\, m\in \Z\}$, then we can see that $L$ is isomorphic to the Heisenberg-Virasoro algebra $\mathcal{H}$ (modulo some center).
            \end{itemize}
\end{remark}
The main results of the present paper are the following theorems.
\begin{theo}\label{theo1.5}Let $\R$ be a Ramond-Block algebra defined in Definition $\ref{def1}$. 
Let $\lambda, \mu\in \C^{*}$ and $a,a',b,b'\in \C$. Then for $\Omega_{\R}(\lambda,a,b):=\C[t^{2}]\oplus t\C[t^{2}]$, we have
\begin{itemize}
  \item[\rm(1)]$\Omega_{\R}(\lambda,a,b)$ is a $\R$-module with the module structure defined in $(\ref{1})$--$(\ref{4})$,
  \item[\rm(2)] $\Omega_{\R}(\lambda,a,b)$ is simple if and only if $a\neq0$ or $b\neq0$,
  \item[\rm(3)] $\Omega_{\R}(\lambda,a,b)\cong \Omega_{\R}(\mu,a',b')$ as $\R$-modules if and only if $\lambda=\mu$, $a=a'$ and $b=b'$,
  \item[\rm(4)] Any free $U(\eta)$-module of rank 1 over $\R$ is isomorphic to $\Omega_{\R}(\lambda,a,b)$ for some $\lambda\in \C^{*}$ and $a,b\in\C$, where $\eta=\C L_{0,0}\oplus\C G_{0,0}$.
\end{itemize}

\end{theo}
\begin{theo}\label{theo1.6}Let $\L$ be a Neveu-Schwarz-Block algebra defined in Definition $\ref{def1}$.  
Let $\lambda,\,\mu\in \C^{*}$ and $a,a',b,b'\in \C$. Then for $\Omega_{\L}(\lambda,a,b):=\C[t]\oplus \C[x]$, we have
\begin{itemize}
  \item[\rm(1)]$\Omega_{\L}(\lambda,a,b)$ is an $\L$-module with the module structure defined in $(\ref{F1})$--$(\ref{F4})$,
  \item[\rm(2)] $\Omega_{\L}(\lambda,a,b)$ is simple if and only if $a\neq0$ or $b\neq0$,
  \item[\rm(3)] $\Omega_{\L}(\lambda,a,b)\cong \Omega_{\L}(\mu,a',b')$ as $\L$-modules if and only if $\lambda=\mu$, $a=a'$ and $b=b'$,
  \item[\rm(4)] Any $\L$-module that is free of rank 2 when restricted to $U(\mathfrak{h})$ is isomorphic to $\Omega_{\L}(\lambda,a,b)$ for some  $\lambda\in \C^{*}$ and $a,b\in\C$, where $\mathfrak{h}=\C L_{0,0}$.
\end{itemize}

\end{theo}
This paper is organized as follows. In section 2, we 
recall some definitions and preliminary results related to  classifications of simple modules over the Witt algebra $\W$ and the Heisenberg-Virasoro algebra $\mathcal{H}$ respectively. In section 3, we construct and classify the free $U(\eta)$-modules of rank $1$ over the Ramond-Block algebra $\R$.
In section 4, we study the Neveu-Schwarz-Block algebra $\L$ and construct the free $U(\mathfrak{h})$-modules of rank 2.
Then a classification of these modules over $\L$ is given in section 5.

Throughout this paper, we denote by $\C$, $\C^{*}$, $\Z$, $\Z_{+}$ and $\N$ the sets of complex numbers, nonzero complex numbers, integers, non-negative integers and positive integers respectively.

\section{Preliminaries}
We begin by briefly introducing our conventions. In this paper, all vector superspaces (resp. superalgebras, supermodules) $V=V_{\overline{0}}\oplus V_{\overline{1}}$ are defined over $\C$. We always assume that a module $M$ of a superalgebra $L$ is a supermodule, i.e., $L_{\overline{i}}\cdot M_{\overline{j}}\subseteq M_{\overline{i}+\overline{j}}$ for all $\overline{i},\overline{j}\in \Z_{2}$, where $\Z_{2}=\{\overline{0},\overline{1}\}$.



Denote by $\C[t]$ the polynomial algebra over $\C$ in indeterminate $t$. Let $\lambda\in \C^{*}$, $a,b\in \C$.
Let $\W$ be the Witt algebra defined in \eqref{Witt-1}.
Then, for $f(t)\in \C[t]$, $\al\in \Z$ and $j=1,2,3$, the $\W$-module structure on $\Omega(\lambda,a):=\C[t]$ is given by
\begin{equation*}
  L_{\al}f(t)=\lambda^{\al}(t-\al a)f(t-\al).
\end{equation*}
Let  $\mathcal{H}$ be the Heisenberg-Virasoro algebra defined in \eqref{Hei-1}--\eqref{Hei-3}.
Then the $\mathcal{H}$-module structure on $\Omega(\lambda,a,b):=\C[t]$ is given by
\begin{eqnarray*}
  L_{\al}f(t)=\lambda^{\al}(t-\al a)f(t-\al),\ \ \  I_{\al}f(t)=b\lambda^{\al}f(t-\al)\ \ \mbox{and} \ \ C_{j}f(t)=0.
\end{eqnarray*}

Keep the notations as above, we need the following results on $\Omega(\lambda,a)$ and $\Omega(\lambda,a,b)$ respectively for later use. 
\begin{theo}\label{lemma-123}\rm{(see [\ref{LZ}, \ref{TZ}]).} {\it Any $\W$-module that is free of rank 1 when restricted to $U(\C L_{0})$ is isomorphic to $\Omega(\lambda,a)$ for some $\lambda\in \C^{*}$ and $a\in\C$. Moreover, $\Omega(\lambda,a)$ is simple if and only if $a \in \C^{*}$ and  $\Omega(\lambda,a)\cong \Omega(\lambda',a')$ as $\W$-modules if and only if $\lambda=\lambda'$ and $a=a'$.}
\end{theo}

\begin{theo}\label{lemma-456}\rm{(see [\ref{CG}])}. {\it Any $\mathcal{H}$-module that is free of rank 1 when restricted to $U(\C L_{0})$ is isomorphic to $\Omega(\lambda,a,b)$ for some $\lambda\in \C^{*}$ and $a,b\in\C$. Moreover, $\Omega(\lambda,a,b)$ is simple if and only if $a\neq0$ or $b\neq0$ and $\Omega(\lambda,a,b)\cong \Omega(\lambda',a',b')$ as $\mathcal{H}$-modules if and only if $\lambda=\lambda'$, $a=a'$ and $b=b'$.}
\end{theo}

\begin{theo}\label{lemma-3.4a} \rm{(see [\ref{YYX}])}. {\it Let $L=L_{\overline{0}}\oplus L_{\overline{1}}$ be a Lie superalgebra with a subalgebra $\eta\subseteq L_{\overline{0}}$, and $[L_{\overline{1}},L_{\overline{1}}]=L_{\overline{0}}$. Then there do not exist $L$-modules which are free of rank 1 as $U(\eta)$-modules.}
\end{theo}

\section{Proof of Theorem \ref{theo1.5}}
For $\lambda\in \C^{*}$ and $a,b\in \C$, we define $\Omega_{\R}(\lambda,a,b)=\C[t^{2}]\oplus t\C[t^{2}]$. Then $\Omega_{\R}(\lambda,a,b)$ is a $\Z_{2}$-graded vector space with $\Omega_{\R}(\lambda,a,b)_{\overline{0}}=\C[t^{2}]$ and $\Omega_{\R}(\lambda,a,b)_{\overline{1}}=t\C[t^{2}]$. The following definition gives a precise construction of a $\mathcal{R}$-module structure on $\Omega_{\R}(\lambda,a,b)$.
\begin{defi}\label{def3.1}
{\rm For $\lambda\in \C^{*}$, $a,b\in \C$, we define the action of $\R$ on $\Omega_{\R}(\lambda,a,b)$ as follows:
\begin{eqnarray}&\!\!\!\!\!\!\!\!\!\!\!\!\!\!\!\!\!\!\!\!\!\!\!\!&\label{1}
L_{m,i}f\left(t^{2}\right)=\lambda^{m}\Big(\delta_{i,0}\left(t^{2}-mqa\right)+\delta_{q,-1}\delta_{i,1}b\Big)f\left(t^{2}-mq\right),\\
&\!\!\!\!\!\!\!\!\!\!\!\!\!\!\!\!\!\!\!\!\!\!\!\!&\label{2}
L_{m,i}tf\big(t^{2}\big)=\lambda^{m}t\Big(\delta_{i,0}\Big(t^{2}-mqa-\frac{mq}{2}\Big)+\delta_{q,-1}\delta_{i,1}b\Big)f\left(t^{2}-mq\right),\\
&\!\!\!\!\!\!\!\!\!\!\!\!\!\!\!\!\!\!\!\!\!\!\!\!&\label{3}
G_{m,i}f\left(t^{2}\right)= \lambda^{m}\delta_{i,0}tf\left(t^{2}-mq\right),\\
&\!\!\!\!\!\!\!\!\!\!\!\!\!\!\!\!\!\!\!\!\!\!\!\!&\label{4}
G_{m,i}tf\left(t^{2}\right)=q\lambda^{m}\Big(\delta_{i,0}\left(t^{2}-2mqa\right)+2\delta_{q,-1}\delta_{i,1}b\Big)f\left(t^{2}-mq\right),
\end{eqnarray}
where $f\left(t^{2}\right)\in \C[t^{2}]$, $m\in \Z$ and $i\in \Z_{+}$.}
\end{defi}

Now we need to show that $\Omega_{\R}(\lambda,a,b)$ is a $\R$-module, namely, the actions of all $L_{m,i}$, $G_{n,j}$ on $\Omega_{\R}(\lambda,a,b)$ satisfy the relations in $\R$. 
\vskip6pt

\noindent{\it Proof of Theorem $\ref{theo1.5}$} (1). 
Let $m,n\in \Z$ and $i,j\in \Z_{+}$. For convenience, we set $d_{i}=\delta_{q,-1}\delta_{i,1}b$. It follows from definitions ($\ref{1}$) and $(\ref{2})$ that
\begin{eqnarray*}&\!\!\!\!\!\!\!\!\!\!\!\!\!\!\!&
 L_{m,i}L_{n,j}f\left(t^{2}\right)=L_{m,i}\lambda^{n}\Big(\delta_{j,0}\left(t^{2}-nqa\right)+d_{j}\Big)f\left(t^{2}-nq\right)
 \nonumber\\\!\!\!\!\!\!\!\!\!\!\!\!&\!\!\!\!\!\!\!\!\!\!\!\!\!\!\!=&
\lambda^{m+n}\Big(\delta_{i,0}\left(t^{2}-mqa\right)+d_{i}\Big)\Big(\delta_{j,0}\left(t^{2}-nqa-mq\right)
+d_{j}\Big)f\left(t^{2}-(m+n)q\right),
 \nonumber\\[6pt]\!\!\!\!\!\!\!\!\!\!\!\!&\!\!\!\!\!\!\!\!\!\!\!\!\!\!\!&
L_{m,i}L_{n,j}tf\left(t^{2}\right)=L_{m,i}\lambda^{n}t\Big(\delta_{j,0}\left(t^{2}-nqa-\frac{nq}{2}\right)+d_{j}\Big)f\left(t^{2}-nq\right)
 \nonumber\\\!\!\!\!\!\!\!\!\!\!\!\!&\!\!\!\!\!\!\!\!\!\!\!\!\!\!\!=&
\lambda^{m+n}t\Big(\delta_{i,0}\left(t^{2}{\sc\!}-{\sc\!}mqa
{\sc\!}-{\sc\!}\frac{mq}{2}\right)
{\sc\!}+{\sc\!}d_{i}\Big)\Big(\delta_{j,0}\left(t^{2}
{\sc\!}-{\sc\!}nqa{\sc\!}-{\sc\!}\frac{nq}{2}{\sc\!}-{\sc\!}mq\right)
{\sc\!}+{\sc\!}d_{j}\Big)f\left(t^{2}{\sc\!}-{\sc\!}(m{\sc\!}+{\sc\!}n)q\right).
\end{eqnarray*}
Then by a little lengthy but straightforward computation, we can get
\begin{eqnarray*}\!\!\!\!\!\!\!\!\!\!\!\!&\!\!\!\!\!\!\!\!\!\!\!\!\!\!\!&
\ \ \ \ \ \ \ \ \ \ \ \ \ \ \ \ L_{m,i}L_{n,j}f\left(t^{2}\right)-L_{n,j}L_{m,i}f\left(t^{2}\right)
\nonumber\\\!\!\!\!\!\!\!\!\!\!\!\!&\!\!\!\!\!\!\!\!\!\!\!\!\!\!\!&
 \ \ \ \ \ \ \ \
\nonumber\\\!\!\!\!\!\!\!\!\!\!\!\!&\!\!\!\!\!\!\!\!\!\!\!\!\!\!\!&
 \ \ \ \ \ \ \ \ \ \
 \ \ \, =\lambda^{m+n}\Big(\delta_{i+j,0}\big(nq-mq\big)\left(t^{2}-(m+n)qa\right)
 +\delta_{i,0}d_{j}nq-\delta_{j,0}d_{i}mq\Big)f\big(t^{2}-(m+n)q\big)
 \nonumber\\\!\!\!\!\!\!\!\!\!\!\!\!&\!\!\!\!\!\!\!\!\!\!\!\!\!\!\!&
 \ \ \ \ \ \ \ \ \ \
 \ \ \, =\big(n(i+q)-m(j+q)\big)\lambda^{m+n}\Big(\delta_{i+j,0}\big(t^{2}-(m+n)qa\big)+d_{i+j}\Big)f\big(t^{2}-(m+n)q\big)
 \nonumber\\\!\!\!\!\!\!\!\!\!\!\!\!&\!\!\!\!\!\!\!\!\!\!\!\!\!\!\!&
 \ \ \ \ \ \ \ \ \ \
\ \ \, =\big(n(i+q)-m(j+q)\big)L_{m+n,i+j}f\big(t^{2}\big)
=[L_{m,i},L_{n,j}]f\big(t^{2}\big),
\end{eqnarray*}
where the third equality is obtained by definition \eqref{1}, and the last equality follows from relation \eqref{J_a_b_}.
Similarly,
\begin{eqnarray*}\!\!\!\!\!\!\!\!\!\!\!\!\!\!\!\!\!\!\!\!\!\!\!\!
&\!\!\!\!\!\!\!\!\!\!\!\!\!\!\!&
\ \ \ \ \ \ \ \ \ \ \ \ \ \
 \ \ \ \ \ \,
 L_{m,i}L_{n,j}tf\left(t^{2}\right)-L_{n,j}L_{m,i}tf\left(t^{2}\right)
\nonumber\\\!\!\!\!\!\!\!\!\!\!\!\!\!\!\!\!\!\!\!\!\!\!\!\!&\!\!\!\!\!\!\!\!\!\!\!\!\!\!\!&
 \ \ \ \ \ \ \ \ \ \ \ \ \ \ \ \
=\lambda^{m+n}t\left(\delta_{i+j,0}\big(nq{\sc\!}-{\sc\!}
mq\big)\Big(t^{2}{\sc\!}-{\sc\!}\big(m{\sc\!}+{\sc\!}n\big)\Big(a
{\sc\!}+{\sc\!}\frac{1}{2}\Big)q\Big)
{\sc\!} +{\sc\!}\delta_{i,0}d_{j}nq
{\sc\!}-{\sc\!}\delta_{j,0}d_{i}mq\right)f\big(t^{2}{\sc\!}-{\sc\!}(m{\sc\!}+{\sc\!}n)q\big)
 \nonumber\\\!\!\!\!\!\!\!\!\!\!\!\!\!\!\!\!\!\!\!\!\!\!\!\!&\!\!\!\!\!\!\!\!\!\!\!\!\!\!\!&
 \ \ \ \ \ \ \ \ \ \ \ \ \ \ \ \
 =\big(n(i+q)-m(j+q)\big)\lambda^{m+n}t\Big(\delta_{i+j,0}\Big(t^{2}-\big(m+n\big)\Big(a+\frac{1}{2}\Big)q\Big)+d_{i+j}\Big)f\big(t^{2}-(m+n)q\big)
 \nonumber\\\!\!\!\!\!\!\!\!\!\!\!\!\!\!\!\!\!\!\!\!\!\!\!\!&\!\!\!\!\!\!\!\!\!\!\!\!\!\!\!&
  \ \ \ \ \ \ \ \ \ \ \ \ \ \ \ \
 =\big(n(i+q)-m(j+q)\big)L_{m+n,i+j}tf\big(t^{2}\big)
 =[L_{m,i},L_{n,j}]tf\big(t^{2}\big),
\end{eqnarray*}
where the third equality is obtained by definition \eqref{2}, and the last equality follows again from relation \eqref{J_a_b_}.
Moreover, using definitions $(\ref{1})$--$(\ref{3})$, we can then obtain the following
\begin{eqnarray*}\!\!\!\!\!\!\!\!\!\!\!\!\!\!\!\!\!\!\!\!\!\!\!\!&\!\!\!\!\!\!\!\!\!\!\!\!\!\!\!&
\ \ \ \ \ \ \ \ \ \ \ \ \ \ \ \ \ \ \ \, L_{m,i}G_{n,j}f\left(t^{2}\right)-G_{n,j}L_{m,i}f\left(t^{2}\right)
\nonumber\\\!\!\!\!\!\!\!\!\!\!\!\!\!\!\!\!\!\!\!\!\!\!\!\!\!\!\!\!&\!\!\!\!\!\!\!\!\!\!\!\!\!\!\!&
 \ \ \ \ \ \ \ \ \ \ \ \ \ \ \ \
=L_{m,i}\lambda^{n}\delta_{j,0}tf\left(t^{2}-nq\right)-G_{n,j}\lambda^{m}\Big(\delta_{i,0}\left(t^{2}-mqa\right)+d_{i}\Big)f\left(t^{2}-mq\right)
 \nonumber\\\!\!\!\!\!\!\!\!\!\!\!\!\!\!\!\!\!\!\!\!\!\!\!\!\!\!\!\!&\!\!\!\!\!\!\!\!\!\!\!\!\!\!\!&
 \ \ \ \ \ \ \ \ \ \ \ \ \ \ \ \
=\lambda^{m+n}\delta_{j,0}t\Big(\delta_{i,0}\left(t^{2}-mqa-\frac{mq}{2}\right)+d_{i}\Big)f\left(t^{2}-(m+n)q\right)
\nonumber\\\!\!\!\!\!\!\!\!\!\!\!\!\!\!\!\!\!\!\!\!\!\!\!\!\!\!\!\!&\!\!\!\!\!\!\!\!\!\!\!\!\!\!\!&
 \ \ \ \ \ \ \ \ \ \ \ \,\ \ \ \  \ \ \ \
-\lambda^{m+n}\delta_{j,0}t\Big(\delta_{i,0}\left(t^{2}-mqa-nq\right)+d_{i}\Big)f\left(t^{2}-(m+n)q\right)
\nonumber\\\!\!\!\!\!\!\!\!\!\!\!\!\!\!\!\!\!\!\!\!\!\!\!\!\!\!\!\!&\!\!\!\!\!\!\!\!\!\!\!\!\!\!\!&
 \ \ \ \ \ \ \ \ \ \ \ \ \ \ \ \
=\lambda^{m+n}\delta_{i,0}\delta_{j,0}\Big(nq-\frac{mq}{2}\Big)tf\big(t^{2}-(m+n)q\big)
 \nonumber\\\!\!\!\!\!\!\!\!\!\!\!\!\!\!\!\!\!\!\!\!\!\!\!\!\!\!\!\!&\!\!\!\!\!\!\!\!\!\!\!\!\!\!\!&
 \ \ \ \ \ \ \ \ \ \ \ \ \ \ \ \
=\Big(n\big(i+q\big)-m\Big(j+\frac{q}{2}\Big)\Big)\lambda^{m+n}\delta_{i+j,0}tf\big(t^{2}-(m+n)q\big)
 \nonumber\\\!\!\!\!\!\!\!\!\!\!\!\!\!\!\!\!\!\!\!\!\!\!\!\!\!\!\!\!&\!\!\!\!\!\!\!\!\!\!\!\!\!\!\!&
 \ \ \ \ \ \ \ \ \ \ \ \ \ \ \ \
=\Big(n\big(i+q\big)-m\Big(j+\frac{q}{2}\Big)\Big)G_{m+n,i+j}f\big(t^{2}\big)
=[L_{m,i},G_{n,j}]f\big(t^{2}\big),
\end{eqnarray*}
where the last equality is yielded by relation $(\ref{J_a_b c_})$. Taking into account relation $(\ref{J_a_b c_})$ again, and according to definitions $(\ref{1})$, $(\ref{2})$ and $(\ref{4})$, we have
\begin{eqnarray*}\!\!\!\!\!\!\!\!\!\!\!\!\!\!\!\!\!\!\!\!\!\!\!\!&\!\!\!\!\!\!\!\!\!\!\!\!\!\!\!&
\ \ \ \ \ \ \ \ \ \ \ \ \ \ \ \ \ \,L_{m,i}G_{n,j}tf\left(t^{2}\right)-G_{n,j}L_{m,i}tf\left(t^{2}\right)
\nonumber\\\!\!\!\!\!\!\!\!\!\!\!\!\!\!\!\!\!\!\!\!\!\!\!\!&\!\!\!\!\!\!\!\!\!\!\!\!\!\!\!&
 \ \ \ \ \ \ \ \ \ \ \ \ \  \
 =L_{m,i}q\lambda^{n}\Big(\delta_{j,0}\left(t^{2}-2nqa\right)+2d_{j}\Big)f\left(t^{2}-nq\right)
 \nonumber\\\!\!\!\!\!\!\!\!\!\!\!\!\!\!\!\!\!\!\!\!\!\!\!\!&\!\!\!\!\!\!\!\!\!\!\!\!\!\!\!&
 \ \ \ \ \ \ \ \ \ \ \ \ \ \ \ \  \
 -G_{n,j}\lambda^{m}t\Big(\delta_{i,0}\left(t^{2}-mqa-\frac{mq}{2}\right)+d_{i}\Big)f\left(t^{2}-mq\right)
 \nonumber\\\!\!\!\!\!\!\!\!\!\!\!\!\!\!\!\!\!\!\!\!\!\!\!\!&\!\!\!\!\!\!\!\!\!\!\!\!\!\!\!&
 \ \ \ \ \ \ \ \ \ \ \ \ \  \
=q\lambda^{m+n}\Big(\delta_{i,0}\left(t^{2}-mqa\right)+d_{i}\Big)\Big(\delta_{j,0}\left(t^{2}-2nqa-mq\right)+2d_{j}\Big)f\left(t^{2}-(m+n)q\right)
\nonumber\\\!\!\!\!\!\!\!\!\!\!\!\!\!\!\!\!\!\!\!\!\!\!\!\!&\!\!\!\!\!\!\!\!\!\!\!\!\!\!\!&
 \ \ \ \ \ \ \ \ \ \ \ \ \ \ \ \  \
 -q\lambda^{m+n}\Big(\delta_{j,0}\big(t^{2}-2nqa\big)+2d_{j}\Big)\Big(\delta_{i,0}\left(t^{2}-mqa-\frac{mq}{2}-nq\right)+d_{i}\Big)f\left(t^{2}-(m+n)q\right)
\nonumber\\\!\!\!\!\!\!\!\!\!\!\!\!\!\!\!\!\!\!\!\!\!\!\!\!&\!\!\!\!\!\!\!\!\!\!\!\!\!\!\!&
 \ \ \ \ \ \ \ \ \ \ \ \ \  \
 =q\lambda^{m+n}\Big(\delta_{i+j,0}\Big(nq{\sc\!}-{\sc\!}
 \frac{mq}{2}\Big)\big(t^{2}{\sc\!}-{\sc\!}2(m{\sc\!}+{\sc\!}
 n)qa\big){\sc\!}+{\sc\!}\big(m{\sc\!}+{\sc\!}2n\big)q\delta_{i,0}d_{j}
 {\sc\!}-{\sc\!}mq\delta_{j,0}d_{i}\Big)f\big(t^{2}{\sc\!}-{\sc\!}(m{\sc\!}+{\sc\!}n)q\big)
 \nonumber\\\!\!\!\!\!\!\!\!\!\!\!\!\!\!\!\!\!\!\!\!\!\!\!\!&\!\!\!\!\!\!\!\!\!\!\!\!\!\!\!&
 \ \ \ \ \ \ \ \ \ \ \ \ \  \
 =\Big(n\big(i+q\big)-m\Big(j+\frac{q}{2}\Big)\Big)q\lambda^{m+n}\Big(\delta_{i+j,0}\left(t^{2}-2(m+n)qa\right)+2d_{i+j}\Big)f\big(t^{2}-(m+n)q\big)
 \nonumber\\\!\!\!\!\!\!\!\!\!\!\!\!\!\!\!\!\!\!\!\!\!\!\!\!&\!\!\!\!\!\!\!\!\!\!\!\!\!\!\!&
 \ \ \ \ \ \ \ \ \ \ \ \ \  \
 =\Big(n\big(i+q\big)-m\Big(j+\frac{q}{2}\Big)\Big)G_{m+n,i+j}tf\big(t^{2}\big)
 =[L_{m,i},G_{n,j}]tf\big(t^{2}\big).
\end{eqnarray*}
Furthermore, with definitions $(\ref{3})$ and $(\ref{4})$, we get
\begin{eqnarray*}\!\!\!\!\!\!\!\!\!\!\!\!&\!\!\!\!\!\!\!\!\!\!\!\!\!\!\!&
\ \ \ \ \ \ \ \ \ \ \ \ \ \ \ \ \ \ \,G_{m,i}G_{n,j}f\left(t^{2}\right)+G_{n,j}G_{m,i}f\left(t^{2}\right)
\nonumber\\\!\!\!\!\!\!\!\!\!\!\!\!&\!\!\!\!\!\!\!\!\!\!\!\!\!\!\!&
 \ \ \ \ \ \ \ \ \ \ \ \ \ \ \
 =G_{m,i}\lambda^{n}\delta_{j,0}tf\left(t^{2}-nq\right)+G_{n,j}\lambda^{m}\delta_{i,0}tf\left(t^{2}-mq\right)
 \nonumber\\\!\!\!\!\!\!\!\!\!\!\!\!&\!\!\!\!\!\!\!\!\!\!\!\!\!\!\!&
\ \ \ \ \ \ \ \ \ \ \ \ \ \ \
=q\lambda^{m+n}\delta_{j,0}\Big(\delta_{i,0}\left(t^{2}-2mqa\right)+2d_{i}\Big)f\left(t^{2}-(m+n)q\right)
\nonumber\\\!\!\!\!\!\!\!\!\!\!\!\!&\!\!\!\!\!\!\!\!\!\!\!\!\!\!\!&
 \ \ \ \ \ \ \ \ \ \ \ \ \ \ \ \ \ \
 +q\lambda^{m+n}\delta_{i,0}\Big(\delta_{j,0}\left(t^{2}-2nqa\right)+2d_{j}\Big)f\left(t^{2}-(m+n)q\right)
\nonumber\\\!\!\!\!\!\!\!\!\!\!\!\!&\!\!\!\!\!\!\!\!\!\!\!\!\!\!\!&
 \ \ \ \ \ \ \ \ \ \ \ \ \ \ \
 =2q\lambda^{m+n}\Big(\delta_{i,0}\delta_{j,0}\big(t^{2}-(m+n)qa\big)+\delta_{j,0}d_{i}+\delta_{i,0}d_{j}\Big)f\big(t^{2}-(m+n)q\big)
 \nonumber\\\!\!\!\!\!\!\!\!\!\!\!\!&\!\!\!\!\!\!\!\!\!\!\!\!\!\!\!&
 \ \ \ \ \ \ \ \ \ \ \ \ \ \ \
 =2q\lambda^{m+n}\Big(\delta_{i+j,0}\big(t^{2}-(m+n)qa\big)+d_{i+j}\Big)f\big(t^{2}-(m+n)q\big)
 =2qL_{m+n,i+j}f\big(t^{2}\big),
\end{eqnarray*}
where the last equality is deduced from definition \eqref{1}, 
and 
\begin{eqnarray*}\!\!\!\!\!\!\!\!\!\!\!\!&\!\!\!\!\!\!\!\!\!\!\!\!\!\!\!&
\ \ \ \ \ \ \ \ \,\ \ \ \ \ \ \ \ \ \,G_{m,i}G_{n,j}tf\left(t^{2}\right)+G_{n,j}G_{m,i}tf\left(t^{2}\right)
\nonumber\\\!\!\!\!\!\!\!\!\!\!\!\!&\!\!\!\!\!\!\!\!\!\!\!\!\!\!\!&
 \ \ \ \ \ \ \ \ \ \ \ \ \ \ \
 =G_{m,i}q\lambda^{n}\left(\delta_{j,0}\big(t^{2}-2nqa\big)+2d_{j}\right)f\left(t^{2}-nq\right)
 \nonumber\\\!\!\!\!\!\!\!\!\!\!\!\!&\!\!\!\!\!\!\!\!\!\!\!\!\!\!\!&
 \ \ \ \ \ \ \ \ \ \ \ \ \ \ \ \ \ \
 +G_{n,j}q\lambda^{m}\left(\delta_{i,0}\big(t^{2}-2mqa\big)+2d_{i}\right)f\left(t^{2}-mq\right)
 \nonumber\\\!\!\!\!\!\!\!\!\!\!\!\!&\!\!\!\!\!\!\!\!\!\!\!\!\!\!\!&
 \ \ \ \ \ \ \ \ \ \ \ \ \ \ \
=q\lambda^{m+n}\delta_{i,0}t\Big(\delta_{j,0}\left(t^{2}-2nqa-mq\right)+2d_{j}\Big)f\left(t^{2}-(m+n)q\right)
\nonumber\\\!\!\!\!\!\!\!\!\!\!\!\!&\!\!\!\!\!\!\!\!\!\!\!\!\!\!\!&
 \ \ \ \ \ \ \ \ \ \ \ \ \ \ \ \ \ \
 +q\lambda^{m+n}\delta_{j,0}t\Big(\delta_{i,0}\left(t^{2}-2mqa-nq\right)+2d_{i}\Big)f\left(t^{2}-(m+n)q\right)
\nonumber\\\!\!\!\!\!\!\!\!\!\!\!\!&\!\!\!\!\!\!\!\!\!\!\!\!\!\!\!&
 \ \ \ \ \ \ \ \ \ \ \ \ \ \ \
 =2q\lambda^{m+n}t\Big(\delta_{i,0}\delta_{j,0}\Big(t^{2}-\big(m+n\big)\Big(a+\frac{1}{2}\Big)q\Big)+\delta_{j,0}d_{i}+\delta_{i,0}d_{j}\Big)f\big(t^{2}-(m+n)q\big)
 \nonumber\\\!\!\!\!\!\!\!\!\!\!\!\!&\!\!\!\!\!\!\!\!\!\!\!\!\!\!\!&
  \ \ \ \ \ \ \ \ \ \ \ \ \ \ \
 =2q\lambda^{m+n}t\Big(\delta_{i+j,0}\Big(t^{2}-\big(m+n\big)\Big(a+\frac{1}{2}\Big)q\Big)+d_{i+j}\Big)f\big(t^{2}-(m+n)q\big)
 \nonumber\\\!\!\!\!\!\!\!\!\!\!\!\!&\!\!\!\!\!\!\!\!\!\!\!\!\!\!\!&
  \ \ \ \ \ \ \ \ \ \ \ \ \ \ \
 =2qL_{m+n,i+j}tf\big(t^{2}\big),
\end{eqnarray*}
where the last equality is implied by definition \eqref{2}.
Therefore $\Omega_{\R}(\lambda,a,b)$ is a $\R$-module and we complete the proof of Theorem $\ref{theo1.5}\,(1)$.
\hfill$\Box$\vskip6pt

Next we want to present the proofs of Theorem $\ref{theo1.5}$ (2) and (3), which provide the sufficient and necessary conditions for the $\R$-module $\Omega_{\R}(\lambda,a,b)$ to be simple, and the isomorphism class of it respectively.
\vskip6pt

\noindent{\it Proof of Theorem $\ref{theo1.5}$} (2) and (3).
Let $N=N_{\overline{0}}\oplus N_{\overline{1}}$ be a nonzero submodule of $\Omega_{\R}(\lambda,a,b)$ for $\lambda\in \C^{*}$ and $a,b\in\C$. Then by $(\ref{4})$, we know that $N_{\overline{0}}\neq 0$. Since $N_{\overline{0}}$ is a submodule of $\Omega_{\R}(\lambda,a,b)_{\overline{0}}$ as $\R_{\overline{0}}$-module, then, it can be regarded as $\W$-module and $\H$-module for $q\neq-1$ and $q=-1$ respectively. Hence, with helps of Theorems $\ref{lemma-123}$ and $\ref{lemma-456}$, we obtain that $N_{\overline{0}}=\Omega_{\R}(\lambda,a,b)_{\overline{0}}$ if and only if $a\neq0$ or $b\neq0$. Furthermore, by $(\ref{3})$, we get $N_{\overline{1}}=\Omega_{\R}(\lambda,a,b)_{\overline{1}}$. Consequently, $N=\Omega_{\R}(\lambda,a,b)$,  
and Theorem $\ref{theo1.5}\,(2)$ holds.

 Suppose $\Omega_{\R}(\lambda,a,b)\cong\Omega_{\R}(\mu,a',b')$ as $\R$-modules. Then, $\Omega_{\R}(\lambda,a,b)_{\overline{0}}\cong\Omega_{\R}(\mu,a',b')_{\overline{0}}$ as $\R_{\overline{0}}$-modules. If $q\neq-1$, by Theorem $\ref{lemma-123}$, we have $\lambda=\mu$ and $a=a'$. If $q=-1$, by Theorem $\ref{lemma-456}$, we have $\lambda=\mu$, $a=a'$ and $b=b'$. Thus, the proof of Theorem $\ref{theo1.5}\,(3)$ is completed.
\hfill$\Box$\vskip6pt

Before proving Theorem $\ref{theo1.5}\,(4)$, we further need the following preliminary results for later use.

Let $M=M_{\overline{0}}\oplus M_{\overline{1}}$ be a $\R$-module such that it is free of rank 1 as a $U(\eta)$-module, where $\eta=\C L_{0,0} \oplus \C G_{0,0}$. Using the
relations $L_{0,0}G_{0,0}=G_{0,0}L_{0,0}$ and $G_{0,0}^{2}=qL_{0,0}$, we obtain that $U(\eta)=\C [L_{0,0}]\oplus G_{0,0}\C [L_{0,0}]$. Take a homogeneous basis element $1\in M$, without loss of generality, up to parity, we may assume $1\in M_{\overline{0}}$, and
\begin{equation}\label{M}
  M=U(\eta)1=\C [L_{0,0}]1\oplus G_{0,0}\C[L_{0,0}]1,
\end{equation}
with $M_{\overline{0}}=\C [L_{0,0}]1$ and $M_{\overline{1}}=G_{0,0}\C [L_{0,0}]1$. Hence, we can suppose that $L_{0,0}1=t^{2}1$ and $G_{0,0}1=t1$.


\begin{lemm}For any $m\in\Z$, $i\in \Z_{+}$ and $r\in \N$, we have
\begin{eqnarray}&\!\!\!\!\!\!\!\!\!\!\!\!\!\!\!\!\!\!\!\!\!\!\!\!&\label{i1}
L_{m,i}(L_{0,0})^{r} =\big(L_{0,0}-mq\big)^{r}L_{m,i},\\
&\!\!\!\!\!\!\!\!\!\!\!\!\!\!\!\!\!\!\!\!\!\!\!\!&\label{i2}
G_{m,i}(L_{0,0})^{r} =\big(L_{0,0}-mq\big)^{r}G_{m,i}.
\end{eqnarray}
\end{lemm}
\begin{proof} According to relation $(\ref{J_a_b_})$, we get
\begin{equation*}
  L_{m,i}L_{0,0}=[L_{m,i},L_{0,0}]+L_{0,0}L_{m,i}=(L_{0,0}-mq)L_{m,i},
\end{equation*}
thus $(\ref{i1})$ is true for $r=1$. Suppose  $(\ref{i1})$ holds for $r$, then we obtain
\begin{equation*}
  L_{m,i}(L_{0,0})^{r+1}=L_{m,i}L_{0,0}(L_{0,0})^{r}=(L_{0,0}-mq)L_{m,i}(L_{0,0})^{r}=\big(L_{0,0}-mq\big)^{r+1}L_{m,i}.
\end{equation*}
Therefore, $(\ref{i1})$ holds for any $r\in \N$.

Similarly, we can get $(\ref{i2})$ by induction on $r$.
\end{proof}

By (\ref{i1}) and a similar proof of [\ref{CY}, Theorem 3.1], we have the following Lemma.
\begin{lemm}\label{lemmab1}For any $m\in\Z$, $i\in \Z_{+}$ and $f(t^{2})\in \C[t^{2}]$, we obtain
\begin{equation}\label{i4}
   L_{m,i}f(t^{2})1=\lambda^{m}\big(\delta_{i,0}\big(t^{2}-mqa\big)+\delta_{q,-1}\delta_{i,1}b\big)f(t^{2}-mq)1.
\end{equation}
\end{lemm}
The following Lemma about the action of $G_{m,i}$ on the homogeneous basis element $1\in M$ with $m\in \Z$ and $i\in \Z_{+}$ is crucial for our further discussion.
\begin{lemm}\label{lemmab2}We have $G_{m,i}1=\lambda^{m}\delta_{i,0}t1$, for any $m\in \Z$ and $i\in \Z_{+}$.
\end{lemm}
\begin{proof}Suppose  $G_{m,i}1=tH_{m,i}(t^{2})1\in t\C[t^{2}]1$ for any $m\in \Z$, $i\in \Z_{+}$ and $H_{m,i}(t^{2})\in \C[t^{2}]$. Then by the relation $G_{0,0}^{2}=qL_{0,0}$ and $(\ref{i2})$, we get
 \begin{equation}\label{j11}
   G_{0,0}G_{m,i}1= G_{0,0}tH_{m,i}(t^{2})1= H_{m,i}(t^{2})G_{0,0}t1= qt^{2}H_{m,i}(t^{2})1.
 \end{equation}
Having in mind that $[G_{m,i},G_{0,0}]=2qL_{m,i}$, and applying $(\ref{i4})$  and $(\ref{j11})$, we obtain
\begin{eqnarray}\label{j12}\nonumber&\!\!\!\!\!\!\!\!\!\!\!\!&
G_{m,i}t1=G_{m,i}G_{0,0}1=2qL_{m,i}1-G_{0,0}G_{m,i}1
\nonumber\\&\!\!\!\!\!\!\!\!\!&
\ \ \ \ \ \ \ \ \
=2q\lambda^{m}\big(\delta_{i,0}\big(t^{2}-mqa\big)+\delta_{q,-1}\delta_{i,1}b\big)1-qt^{2}H_{m,i}(t^{2})1.
\end{eqnarray}
Then using $(\ref{i2})$ and $(\ref{j12})$, we are able to have the following 
\begin{eqnarray*}&\!\!\!\!\!\!\!\!\!\!\!\!&
G_{m,i}G_{m,i}1=G_{m,i}tH_{m,i}\big(t^{2}\big)1=H_{m,i}\big(t^{2}-mq\big)G_{m,i}t1\\
&\!\!\!\!\!\!\!\!\!\!\!\!&
\ \ \ \ \ \ \ \ \ \ \ \ \ \
=H_{m,i}\big(t^{2}-mq\big)\Big(2q\lambda^{m}\big(\delta_{i,0}\big(t^{2}-mqa\big)+\delta_{q,-1}\delta_{i,1}b\big)-qt^{2}H_{m,i}(t^{2})\Big)1.
\end{eqnarray*}
Therefore, combining this result with the relation $$[G_{m,i},G_{m,i}]1=2qL_{2m,2i}1=2q\lambda^{2m}\delta_{2i,0}\big(t^{2}-2mqa\big)1,$$ we can obtain
\begin{equation}\label{eqa5}
  H_{m,i}\big(t^{2}-mq\big)\Big(2\lambda^{m}\Big(\delta_{i,0}\big(t^{2}-mqa\big)+\delta_{q,-1}\delta_{i,1}b\Big)-t^{2}H_{m,i}\left(t^{2}\right)\Big)
  =\lambda^{2m}\delta_{2i,0}\big(t^{2}-2mqa\big).
\end{equation}
Taking $i=0$ in $(\ref{eqa5})$, and comparing the coefficients of $t^{2}$, we have $ H_{m,0}\big(t^{2}\big)=\lambda^{m}$. 
Take $i\in \N$ in $(\ref{eqa5})$, then we get
\begin{equation*}
 H_{m,i}\big(t^{2}-mq\big)\Big(2\lambda^{m}\delta_{q,-1}\delta_{i,1}b-t^{2}H_{m,i}\left(t^{2}\right)\Big)=0.
\end{equation*}
Thus, $H_{m,i}\big(t^{2}\big)=0$ for $i\in \N$. Consequently, $G_{m,i}1=tH_{m,i}(t^{2})1=\lambda^{m}\delta_{i,0}t1$ for any $m\in \Z$ and $i\in \Z_{+}$.
Hence, the proof is completed. 
\end{proof}

Now we are ready to present the proof of Theorem $\ref{theo1.5}\,(4)$, which gives a complete classification of free $U(\eta)$-modules of rank 1 over the Lie superalgebra of Block type $\R$.\vskip6pt

\noindent{\it Proof of Theorem $\ref{theo1.5}\,(4)$.} For any $f\left(t^{2}\right)\in \C[t^{2}]$, $m\in \Z$ and $i\in \Z_{+}$, by Lemma $\ref{lemmab2}$ and $(\ref{i2})$, we get
\begin{eqnarray}\label{j13}&\!\!\!\!\!\!\!\!\!\!\!\!&
G_{m,i}f\left(t^{2}\right)1=G_{m,i}f\big(L_{0,0}\big)1=f\big(L_{0,0}-mq\big)G_{m,i}1=\lambda^{m}\delta_{i,0}tf\big(t^{2}-mq\big)1.
\end{eqnarray}
Moreover, if we take into account $(\ref{i2})$ and $(\ref{i4})$, then from Lemma $\ref{lemmab2}$, we have
\begin{eqnarray}\label{j14}\!\!\!\!\!\!\!\!\!\!\!\!&\!\!\!\!\!\!\!\!\!\!\!\!\!\!\!&
\ \ \ \ \ \ \ \ \ \ \ \ G_{m,i}tf\left(t^{2}\right)1=f\big(t^{2}-mq\big)G_{m,i}t1=f\big(t^{2}-mq\big)G_{m,i}G_{0,0}1
\nonumber\\\!\!\!\!\!\!\!\!\!\!\!\!&\!\!\!\!\!\!\!\!\!\!\!\!\!\!\!&
\ \ \ \ \ \ \ \ \ \ \ \ \ \ \ \ \ \ \ \ \ \ \ \ \ \ \ \ \
=2qf\big(t^{2}-mq\big)L_{m,i}1-f\big(t^{2}-mq\big)G_{0,0}G_{m,i}1
\nonumber\\\!\!\!\!\!\!\!\!\!\!\!\!&\!\!\!\!\!\!\!\!\!\!\!\!\!\!\!&
\ \ \ \ \ \ \ \ \ \ \ \ \ \ \ \ \ \ \ \ \ \ \ \ \ \ \ \ \
=f\big(t^{2}-mq\big)\Big(2q\lambda^{m}\Big(\delta_{i,0}\big(t^{2}-mqa\big)+\delta_{q,-1}\delta_{i,1}b\Big)-q\lambda^{m}\delta_{i,0}t^{2}\Big)1
\nonumber\\\!\!\!\!\!\!\!\!\!\!\!\!&\!\!\!\!\!\!\!\!\!\!\!\!\!\!\!&
\ \ \ \ \ \ \ \ \ \ \ \ \ \ \ \ \ \ \ \ \ \ \ \ \ \ \ \ \ =q\lambda^{m}\Big(\delta_{i,0}\left(t^{2}-2mqa\right)+2\delta_{q,-1}\delta_{i,1}b\Big)f\big(t^{2}-mq\big)1,
\end{eqnarray}
where the third equality follows from relation $[G_{m,i},G_{0,0}]=2qL_{m,i}$.
Furthermore, using $(\ref{i4})$ and $(\ref{j13})$, we can then obtain the following
\begin{eqnarray}\label{j15}\!\!\!\!\!\!\!\!\!\!\!\!\!\!\!\!&\!\!\!\!\!\!\!\!\!\!\!\!\!\!\!&
\ \ \ \ \ \ \ \ \ \ \ \ \ L_{m,i}tf\left(t^{2}\right)1=L_{m,i}G_{0,0}f\left(t^{2}\right)1=G_{0,0}L_{m,i}f\left(t^{2}\right)1-\frac{mq}{2}G_{m,i}f\left(t^{2}\right)1
\nonumber\\\!\!\!\!\!\!\!\!\!\!\!\!\!\!\!\!&\!\!\!\!\!\!\!\!\!\!\!\!\!\!\!&
\ \ \ \ \ \ \ \ \ \ \ \ \ \ \ \ \ \ \ \ \ \ \ \ \ \ \ \ \ \
=G_{0,0}\lambda^{m}\Big(\delta_{i,0}\big(t^{2}-mqa\big)+\delta_{q,-1}\delta_{i,1}b\Big)f\big(t^{2}-mq\big)1
-\frac{mq}{2}\lambda^{m}\delta_{i,0}tf\big(t^{2}-mq\big)1
\nonumber\\\!\!\!\!\!\!\!\!\!\!\!\!\!\!\!\!&\!\!\!\!\!\!\!\!\!\!\!\!\!\!\!&
\ \ \ \ \ \ \ \ \ \ \ \ \ \ \ \ \ \ \ \ \ \ \ \ \ \ \ \ \ \
=\lambda^{m}t\Big(\delta_{i,0}\big(t^{2}-mqa\big)+\delta_{q,-1}\delta_{i,1}b-\frac{mq}{2}\delta_{i,0}\Big)f\big(t^{2}-mq\big)1
\nonumber\\\!\!\!\!\!\!\!\!\!\!\!\!\!\!\!\!&\!\!\!\!\!\!\!\!\!\!\!\!\!\!\!&
\ \ \ \ \ \ \ \ \ \ \ \ \ \ \ \ \ \ \ \ \ \ \ \ \ \ \ \ \ \
=\lambda^{m}t\Big(\delta_{i,0}\Big(t^{2}-mqa-\frac{mq}{2}\Big)+\delta_{q,-1}\delta_{i,1}b\Big)f\big(t^{2}-mq\big)1,
\end{eqnarray}
where from relation $[L_{m,i},G_{0,0}]=-\frac{mq}{2}G_{m,i}$, we can deduce that the second equality holds.

Hence, $(\ref{j13})$--$(\ref{j15})$ together with $(\ref{M})$ and $(\ref{i4})$ imply that $M\cong\Omega_{\R}(\lambda,a,b)$ as $\R$-modules and then we complete the proof. 
\hfill$\Box$

\section{Proof of Theorem \ref{theo1.6} (1), (2) and (3)}
Let $\Omega_{\L}(\lambda,a,b)=\C[t]\oplus \C[x]$. Then $\Omega_{\L}(\lambda,a,b)$ is a $\Z_{2}$-graded vector space with $\Omega_{\L}(\lambda,a,b)_{\overline{0}}=\C[t]$ and $\Omega_{\L}(\lambda,a,b\big)_{\overline{1}}=\C[x]$. The following gives a precise construction of an $\L$-module structure on $\Omega_{\L}(\lambda,a,b)$.
\begin{defi}\label{def4.1}
{\rm For $\lambda\in \C^{*}$, $a,b\in \C$, we define the action of $\L$ on $\Omega_{\L}(\lambda,a,b)$ as follows:
\begin{eqnarray}&\!\!\!\!\!\!\!\!\!\!\!\!\!\!\!\!\!\!\!\!\!\!\!\!&\label{F1}
L_{m,i}f(t)=\lambda^{m}\big(\delta_{i,0}\big(t-mqa\big)+\delta_{q,-1}\delta_{i,1}b\big)f(t-mq),\\
&\!\!\!\!\!\!\!\!\!\!\!\!\!\!\!\!\!\!\!\!\!\!\!\!&\label{F2}
L_{m,i}g(x)=\lambda^{m}\Big(\delta_{i,0}\Big(x-mqa-\frac{mq}{2}\Big)+\delta_{q,-1}\delta_{i,1}b\Big)g(x-mq),\\
&\!\!\!\!\!\!\!\!\!\!\!\!\!\!\!\!\!\!\!\!\!\!\!\!&\label{F3}
G_{r,i}f(t)= \lambda^{r-\frac{1}{2}}\delta_{i,0}f(x-rq),\\
&\!\!\!\!\!\!\!\!\!\!\!\!\!\!\!\!\!\!\!\!\!\!\!\!&\label{F4}
G_{r,i}g(x)=q\lambda^{r+\frac{1}{2}}\big(\delta_{i,0}(t-2rqa)+2\delta_{q,-1}\delta_{i,1}b\big)g(t-rq),
\end{eqnarray}
where $f(t)\in \C[t]$, $g(x)\in \C[x]$, $m\in \Z$, $r\in \frac{1}{2}+\Z$ and $i\in \Z_{+}$.}
\end{defi}
According to the above definition, firstly we will show that $\Omega_{\L}(\lambda,a,b)$ is an $\L$-module, namely, the actions of all $L_{m,i}$, $G_{l,j}$ on $\Omega_{\L}(\lambda,a,b)$ satisfy the relations in $\L$.\vskip6pt

\noindent{\it Proof of Theorem $\ref{theo1.6}$} (1). Let $m,n\in \Z$, $r,l\in \frac{1}{2}+\Z$ and $i,j\in \Z_{+}$. For convenience, we set $d_{i}=\delta_{q,-1}\delta_{i,1}b$. By definitions $(\ref{F1})$ and $(\ref{F2})$, and a similar proof of Theorem $\ref{theo1.5}$ (1), we get
\begin{eqnarray*}&\!\!\!\!\!\!\!\!\!\!\!\!\!\!\!\!\!\!\!\!\!\!\!\!&
[L_{m,i},L_{n,j}]f(t)=L_{m,i}L_{n,j}f(t)-L_{n,j}L_{m,i}f(t), \\
&\!\!\!\!\!\!\!\!\!\!\!\!\!\!\!\!\!\!\!\!\!\!\!\!&
  [L_{m,i},L_{n,j}]g(x)=L_{m,i}L_{n,j}g(x)-L_{n,j}L_{m,i}g(x).
\end{eqnarray*}
Moreover, using definitions $(\ref{F1})$--$(\ref{F3})$, we obtain
\begin{eqnarray*}\!\!\!\!\!\!\!\!\!\!\!\!&\!\!\!\!\!\!\!\!\!\!\!\!\!\!\!&
\ \ \ \ \ \ \ \ \ \ \ \ \ \ \ \ \ \ \ \ \,L_{m,i}G_{l,j}f(t)-G_{l,j}L_{m,i}f(t)
\nonumber\\\!\!\!\!\!\!\!\!\!\!\!\!&\!\!\!\!\!\!\!\!\!\!\!\!\!\!\!&
\ \ \ \ \ \ \ \ \ \ \ \ \ \ \ \ \
=L_{m,i}\lambda^{l-\frac{1}{2}}\delta_{j,0}f\big(x-lq\big)-G_{l,j}\lambda^{m}\big(\delta_{i,0}\big(t-mqa\big)+d_{i}\big)f\big(t-mq\big)
\nonumber\\\!\!\!\!\!\!\!\!\!\!\!\!&\!\!\!\!\!\!\!\!\!\!\!\!\!\!\!&
\ \ \ \ \ \ \ \ \ \ \ \ \ \ \ \ \
=\lambda^{m+l-\frac{1}{2}}\delta_{j,0}\Big(\delta_{i,0}\Big(x-mqa-\frac{mq}{2}\Big)+d_{i}\Big)f\big(x-lq-mq\big)
\nonumber\\\!\!\!\!\!\!\!\!\!\!\!\!&\!\!\!\!\!\!\!\!\!\!\!\!\!\!\!&
\ \ \ \ \ \ \ \ \ \ \ \ \ \ \ \ \ \ \ \
-\lambda^{m+l-\frac{1}{2}}\delta_{j,0}\Big(\delta_{i,0}\big(x-mqa-lq\big)+d_{i}\Big)f\big(x-lq-mq\big)
\nonumber\\\!\!\!\!\!\!\!\!\!\!\!\!&\!\!\!\!\!\!\!\!\!\!\!\!\!\!\!&
\ \ \ \ \ \ \ \ \ \ \ \ \ \ \ \ \
=\Big(lq-\frac{mq}{2}\Big)\lambda^{m+l-\frac{1}{2}}\delta_{i+j,0}f\big(x-(l+m)q\big)
\nonumber\\\!\!\!\!\!\!\!\!\!\!\!\!&\!\!\!\!\!\!\!\!\!\!\!\!\!\!\!&
\ \ \ \ \ \ \ \ \ \ \ \ \ \ \ \ \
=\Big(l\big(i+q\big)-m\Big(j+\frac{q}{2}\Big)\Big)G_{m+l,i+j}f(t)
=[L_{m,i},G_{l,j}]f(t),
\end{eqnarray*}
where the last equality is obtained by relation $(\ref{J_a_b c_})$.
With definitions $(\ref{F1})$, $(\ref{F2})$ and $(\ref{F4})$, we can have the following
\begin{eqnarray*}\!\!\!\!\!\!\!\!\!\!\!\!\!\!\!\!\!\!\!\!\!\!\!\!&\!\!\!\!\!\!\!\!\!\!\!\!\!\!\!&
\ \ \ \ \ \ \ \ \ \ \ \ \ \ \ \ \ \ \ L_{m,i}G_{l,j}g(x)-G_{l,j}L_{m,i}g(x)
\nonumber\\\!\!\!\!\!\!\!\!\!\!\!\!\!\!\!\!\!\!\!\!\!\!\!\!&\!\!\!\!\!\!\!\!\!\!\!\!\!\!\!&
\ \ \ \ \ \ \ \ \ \ \ \ \ \ \
=L_{m,i}q\lambda^{l+\frac{1}{2}}\big(\delta_{j,0}(t-2lqa)+2d_{j}\big)g\big(t-lq\big)
-G_{l,j}\lambda^{m}\Big(\delta_{i,0}\Big(x-mqa-\frac{mq}{2}\Big)+d_{i}\Big)g\big(x-mq\big)
\nonumber\\\!\!\!\!\!\!\!\!\!\!\!\!\!\!\!\!\!\!\!\!\!\!\!\!&\!\!\!\!\!\!\!\!\!\!\!\!\!\!\!&
\ \ \ \ \ \ \ \ \ \ \ \ \ \ \
=q\lambda^{m+l+\frac{1}{2}}\Big(\delta_{i,0}\big(t-mqa\big)+d_{i}\Big)\Big(\delta_{j,0}\big(t-2lqa-mq\big)+2d_{j}\Big)g\big(t-lq-mq\big)
\nonumber\\\!\!\!\!\!\!\!\!\!\!\!\!\!\!\!\!\!\!\!\!\!\!\!\!&\!\!\!\!\!\!\!\!\!\!\!\!\!\!\!&
\ \ \ \ \ \ \ \ \ \ \ \ \ \ \ \ \ \
-q\lambda^{m+l+\frac{1}{2}}\Big(\delta_{j,0}\big(t-2lqa\big)+2d_{j}\Big)\Big(\delta_{i,0}\Big(t-mqa-\frac{mq}{2}-lq\Big)+d_{i}\Big)g\big(t-lq-mq\big)
\nonumber\\\!\!\!\!\!\!\!\!\!\!\!\!\!\!\!\!\!\!\!\!\!\!\!\!&\!\!\!\!\!\!\!\!\!\!\!\!\!\!\!&
\ \ \ \ \ \ \ \ \ \ \ \ \ \ \
=q\lambda^{m+l+\frac{1}{2}}\Big(\delta_{i+j,0}\Big(lq-\frac{mq}{2}\Big)\big(t-2(m+l)qa\big)-\delta_{j,0}d_{i}mq+\big(2l+m\big)q\delta_{i,0}d_{j}\Big)g\big(t-(l+m)q\big)
\nonumber\\\!\!\!\!\!\!\!\!\!\!\!\!\!\!\!\!\!\!\!\!\!\!\!\!&\!\!\!\!\!\!\!\!\!\!\!\!\!\!\!&
\ \ \ \ \ \ \ \ \ \ \ \ \ \ \
=\Big(l\big(i+q\big)-m\Big(j+\frac{q}{2}\Big)\Big)q\lambda^{m+l+\frac{1}{2}}\Big(\delta_{i+j,0}\big(t-2(m+l)qa\big)+2d_{i+j}\Big)g\big(t-(l+m)q\big)
\nonumber\\\!\!\!\!\!\!\!\!\!\!\!\!\!\!\!\!\!\!\!\!\!\!\!\!&\!\!\!\!\!\!\!\!\!\!\!\!\!\!\!&
\ \ \ \ \ \ \ \ \ \ \ \ \ \ \
=\Big(l\big(i+q\big)-m\Big(j+\frac{q}{2}\Big)\Big)G_{m+l,i+j}g(x)
=[L_{m,i},G_{l,j}]g(x),
\end{eqnarray*}
where the last equality follows again from relation $(\ref{J_a_b c_})$.
Furthermore, due to definitions $(\ref{F3})$ and $(\ref{F4})$, it follows that
\begin{eqnarray*}\!\!\!\!\!\!\!\!\!\!\!\!&\!\!\!\!\!\!\!\!\!\!\!\!\!\!\!&
\ \ \ \ \ \ \ \ \ \ \ \ \ G_{l,i}G_{r,j}f(t)=G_{l,i}\lambda^{r-\frac{1}{2}}\delta_{j,0}f(x-rq)
\nonumber\\\!\!\!\!\!\!\!\!\!\!\!\!&\!\!\!\!\!\!\!\!\!\!\!\!\!\!\!&
\ \ \ \ \ \ \ \ \ \ \ \ \ \ \ \ \ \ \ \ \ \ \ \ \ \ \ \ \,
=q\lambda^{r+l}\delta_{j,0}\Big(\delta_{i,0}\big(t-2lqa\big)+2d_{i}\Big)f\big(t-lq-rq\big),
\\[4pt]
\!\!\!\!\!\!\!\!\!\!\!\!&\!\!\!\!\!\!\!\!\!\!\!\!\!\!\!&
\ \ \ \ \ \ \ \ \ \ \ \ \
G_{l,i}G_{r,j}g(x)=G_{l,i}q\lambda^{r+\frac{1}{2}}\Big(\delta_{j,0}\big(t-2rqa\big)+2d_{j}\Big)g\big(t-rq\big)
 \nonumber\\\!\!\!\!\!\!\!\!\!\!\!\!&\!\!\!\!\!\!\!\!\!\!\!\!\!\!\!&
\ \ \ \ \ \ \ \ \ \ \ \ \ \ \ \ \ \ \ \ \ \ \ \ \ \ \ \ \
 =q\lambda^{r+l}\delta_{i,0}\Big(\delta_{j,0}\big(x-2rqa-lq\big)+2d_{j}\Big)g\big(x-lq-rq\big).
\end{eqnarray*}
Then, by a little lengthy but straightforward computation, we can obtain 
\begin{eqnarray*}\!\!\!\!\!\!\!\!\!\!\!\!&\!\!\!\!\!\!\!\!\!\!\!\!\!\!\!&
\ \ \ \ \ \ \ \ \ \  \ \ \ \ \ \ \ \ \ G_{l,i}G_{r,j}f(t)+G_{r,j}G_{l,i}f(t)
\nonumber\\\!\!\!\!\!\!\!\!\!\!\!\!&\!\!\!\!\!\!\!\!\!\!\!\!\!\!\!&
\ \ \ \ \ \ \ \ \ \ \ \ \ \ \ \
=q\lambda^{r+l}\Big(\delta_{i,0}\delta_{j,0}\big(2t-2(l+r)qa\big)+2\delta_{j,0}d_{i}+2\delta_{i,0}d_{j}\Big)f\big(t-(l+r)q\big)
\nonumber\\\!\!\!\!\!\!\!\!\!\!\!\!&\!\!\!\!\!\!\!\!\!\!\!\!\!\!\!&
\ \ \ \ \ \ \ \ \ \ \ \ \ \ \ \
=2q\lambda^{r+l}\Big(\delta_{i+j,0}\big(t-(l+r)qa\big)+d_{i+j}\Big)f\big(t-(l+r)q\big)
\nonumber\\\!\!\!\!\!\!\!\!\!\!\!\!&\!\!\!\!\!\!\!\!\!\!\!\!\!\!\!&
\ \ \ \ \ \ \ \ \ \ \ \ \ \ \ \
=2qL_{r+l,i+j}f(t)
=[G_{l,i},G_{r,j}]f(t),
\end{eqnarray*}
where the third equality is yielded by definition \eqref{F1}, and the last equality is obtained by relation \eqref{J_a_b_c d}. A similar computation yields that

\begin{eqnarray*}\!\!\!\!\!\!\!\!\!\!\!\!&\!\!\!\!\!\!\!\!\!\!\!\!\!\!\!&
\ \ \ \ \ \ \ \ \ \  \ \ \ \ \ \ \ \ \ \ \ \ \ \ \ G_{l,i}G_{r,j}g(x)+G_{r,j}G_{l,i}g(x)
\nonumber\\\!\!\!\!\!\!\!\!\!\!\!\!&\!\!\!\!\!\!\!\!\!\!\!\!\!\!\!&
\ \ \ \ \ \ \ \ \ \ \ \ \ \ \ \ \ \ \ \ \ \
=q\lambda^{r+l}\Big(\delta_{i,0}\delta_{j,0}\big(2x-(2a+1)(l+r)q\big)+2\delta_{j,0}d_{i}+2\delta_{i,0}d_{j}\Big)g\big(x-(l+r)q\big)
\nonumber\\\!\!\!\!\!\!\!\!\!\!\!\!&\!\!\!\!\!\!\!\!\!\!\!\!\!\!\!&
\ \ \ \ \ \ \ \ \ \ \ \ \ \ \ \ \ \ \ \ \ \
=2q\lambda^{r+l}\Big(\delta_{i+j,0}\Big(x-\Big(a+\frac{1}{2}\Big)(l+r)q\Big)+d_{i+j}\Big)g\big(x-(l+r)q\big)
\nonumber\\\!\!\!\!\!\!\!\!\!\!\!\!&\!\!\!\!\!\!\!\!\!\!\!\!\!\!\!&
\ \ \ \ \ \ \ \ \ \ \ \ \ \ \ \ \ \ \ \ \ \
=2qL_{r+l,i+j}g(x)
=[G_{l,i},G_{r,j}]g(x),
\end{eqnarray*}
where the third equality is implied by definition \eqref{F2}, and the last equality follows again from relation \eqref{J_a_b_c d}.
Therefore, $\Omega_{\L}(\lambda,a,b)$ is an $\L$-module, and  Theorem $\ref{theo1.6}$ (1) follows.
\hfill$\Box$\vskip6pt

Before proving Theorem $\ref{theo1.6}$ (2) and (3), we first present the following results.
Let $\tau:\L\rightarrow \R$ be a linear map defined by
\begin{equation*}
  \tau(L_{m,i})=\frac{1}{2}L_{2m,i},\ \ \ \ \ \ \ \tau(G_{r,j})=\frac{1}{\sqrt{2}}G_{2r,j},
\end{equation*}
for $m\in \Z$, $r\in \frac{1}{2}+\Z$ and $i,j\in \Z_{+}$. It is straightforward to verify that $\tau$ is an injective Lie superalgebra homomorphism. Then $\L$ can be regarded as a subalgebra of $\R$. In particular, $\Omega_{\R}(\lambda,a,b)$ is an $\L$-module.
\begin{prop}\label{prop-3.2a} For $\lambda\in\C^{*}$ and $a,b\in\C$, the $\L$-module $\Omega_{\R}(\lambda,a,b)$ is simple  if and only if it is a simple $\R$-module.
\end{prop}
\begin{proof}It is sufficient to prove that $\Omega_{\R}(\lambda,a,b)$ is a simple $\L$-module if $a\neq0$ or $b\neq0$. Let $N=N_{\overline{0}}\oplus N_{\overline{1}}$ be a nonzero submodule of $\Omega_{\R}(\lambda,a,b)$ as an $\L$-module. From  the fact $$\tau(G_{r,i})tf\left(t^{2}\right)=\frac{q}{\sqrt{2}}\lambda^{2r}\Big(\delta_{i,0}\left(t^{2}-4rqa\right)
+2\delta_{q,-1}\delta_{i,1}b\Big)f\left(t^{2}-2rq\right),$$ we can see that $N_{\overline{0}}\neq0$. Since $$\tau(L_{m,i})f\left(t^{2}\right)=\frac{1}{2}\lambda^{2m}\Big(\delta_{i,0}\left(t^{2}-2mqa\right)
+\delta_{q,-1}\delta_{i,1}b\Big)f\left(t^{2}-2mq\right)$$ for all $m\in \Z$ and $i\in \Z_{+}$, with helps of Theorems $\ref{lemma-123}$ and $\ref{lemma-456}$, we get that $\Omega_{\R}(\lambda,a,b)_{\overline{0}}$ is a simple $\tau(\W)$-module and simple $\tau(\H)$-module for $q\neq-1$ and $q=-1$ respectively. Thus, $N_{\overline{0}}=\Omega_{\R}(\lambda,a,b)_{\overline{0}}$.

Moreover, by the fact that $\tau(G_{r,j})f\left(t^{2}\right)=\frac{1}{\sqrt{2}}\lambda^{2r}\delta_{j,0}tf\left(t^{2}-2rq\right)\subseteq N_{\overline{1}}$ for any $r\in \frac{1}{2}+\Z$, $j\in \Z_{+}$ and $f\left(t^{2}\right)\in \C[t^{2}]$, we obtain $N_{\overline{1}}=\Omega_{\R}(\lambda,a,b)_{\overline{1}}$. Therefore, $N=\Omega_{\R}(\lambda,a,b)$.
\end{proof}

We observe that there is  a close connection between the $\R$-modules and the $\L$-modules constructed in Definitions \ref{def3.1} and \ref{def4.1} respectively.
\begin{prop}\label{prop-3.3a}Let $\lambda\in \C^{*}$ and $a,b\in \C$. Then $\Omega_{\L}(\lambda,a,b)\cong \Omega_{\R}(\lambda^{\frac{1}{2}},a,2b)$ as $\L$-modules.
\end{prop}
\begin{proof}Let
$\phi: \Omega_{\L}(\lambda,a,b)\rightarrow\Omega_{\R}(\lambda^{\frac{1}{2}},a,2b)$ be the linear map defined by sending,
\begin{eqnarray*}
 f(t)&\longmapsto& f\left(\frac{t^{2}}{2}\right),\ \ \ \ \ \ \ \
 g(x)\longmapsto \frac{\lambda^{\frac{1}{2}}}{\sqrt{2}}tg\left(\frac{t^{2}}{2}\right).
 \end{eqnarray*}
It is straightforward to prove that $\phi$ is bijective.
 In the following, we will show that it is an $\L$-module homomorphism.

 For any $m\in \Z$, $r\in \frac{1}{2}+\Z$, $i,j\in \Z_{+}$, $f(t)\in \C[t]$ and $g(x)\in \C[x]$, we set $d_{i}=\delta_{q,-1}\delta_{i,1}b$. On the one hand, from definitions $(\ref{F1})$--$(\ref{F4})$, we get the following
 \begin{eqnarray*}&\!\!\!\!\!\!\!\!\!\!\!\!\!\!\!\!\!\!\!\!\!\!\!\!&
 \ \ \ \ \ \ \ \ \ \ \ \ \ \ \ \ \ \phi\big(L_{m,i}f(t)\big) = \phi\Big(\lambda^{m}\big(\delta_{i,0}\big(t-mqa\big)+d_{i}\big)f\big(t-mq\big)\Big) \\
  &\!\!\!\!\!\!\!\!\!\!\!\!\!\!\!\!\!\!\!\!\!\!\!\!&
 \ \ \ \ \ \ \ \ \ \ \ \ \ \ \ \ \ \ \ \ \ \ \ \ \ \ \ \ \ \ \ \ \
=\lambda^{m}\Big(\delta_{i,0}\Big(\frac{t^{2}}{2}-mqa\Big)+d_{i}\Big)f\Big(\frac{t^{2}}{2}-mq\Big),\\
&\!\!\!\!\!\!\!\!\!\!\!\!\!\!\!\!\!\!\!\!\!\!\!\!&
 \ \ \ \ \ \ \ \ \ \ \ \ \ \ \ \ \ \phi\big(L_{m,i}g(x)\big) = \phi\Big(\lambda^{m}\Big(\delta_{i,0}\Big(x-mqa-\frac{mq}{2}\Big)+d_{i}\Big)g\big(x-mq\big)\Big) \\
  &\!\!\!\!\!\!\!\!\!\!\!\!\!\!\!\!\!\!\!\!\!\!\!\!&
 \ \ \ \ \ \ \ \ \ \ \ \ \ \ \ \ \ \ \ \ \ \ \ \ \ \ \ \ \ \ \ \ \ \,
=\frac{\lambda^{m+\frac{1}{2}}}{\sqrt{2}}t\Big(\delta_{i,0}\Big(\frac{t^{2}}{2}-mqa-\frac{mq}{2}\Big)+d_{i}\Big)
g\Big(\frac{t^{2}}{2}-mq\Big),\\
 &\!\!\!\!\!\!\!\!\!\!\!\!\!\!\!\!\!\!\!\!\!\!\!\!&
 \ \ \ \ \ \ \ \ \ \ \ \ \ \ \ \ \ \ \phi\big(G_{r,j}f(t)\big) = \phi\Big(\lambda^{r-\frac{1}{2}}\delta_{j,0}f\big(x-rq\big)\Big)=\frac{\lambda^{r}}{\sqrt{2}}\delta_{j,0}tf\Big(\frac{t^{2}}{2}-rq\Big),\\
&\!\!\!\!\!\!\!\!\!\!\!\!\!\!\!\!\!\!\!\!\!\!\!\!&
 \ \ \ \ \ \ \ \ \ \ \ \ \ \ \ \ \ \ \phi\big(G_{r,j}g(x)\big)=\phi\Big(q\lambda^{r+\frac{1}{2}}\big(\delta_{j,0}(t-2rqa)+2d_{j}\big)g\big(t-rq\big)\Big)\\
&\!\!\!\!\!\!\!\!\!\!\!\!\!\!\!\!\!\!\!\!\!\!\!\!&
\ \ \ \ \ \ \ \ \ \ \ \ \ \ \  \ \ \ \ \ \ \ \ \ \ \ \ \ \ \ \ \ \ \
=q\lambda^{r+\frac{1}{2}}\Big(\delta_{j,0}\Big(\frac{t^{2}}{2}-2rqa\Big)+2d_{j}\Big)g\Big(\frac{t^{2}}{2}-rq\Big),
\end{eqnarray*}
while on the other hand, by definitions $(\ref{1})$--$(\ref{4})$, we obtain
\begin{eqnarray*}&\!\!\!\!\!\!\!\!\!\!\!\!\!\!\!\!\!\!\!\!\!\!\!\!&
 \ \ \ \ \ \ \ \ \ \ \ \ \ \ \ \tau(L_{m,i})\phi\big(f(t)\big)=\frac{1}{2}L_{2m,i}f\Big(\frac{t^{2}}{2}\Big)
=\lambda^{m}\Big(\delta_{i,0}\Big(\frac{t^{2}}{2}-mqa\Big)+d_{i}\Big)f\Big(\frac{t^{2}}{2}-mq\Big), \\
&\!\!\!\!\!\!\!\!\!\!\!\!\!\!\!\!\!\!\!\!\!\!\!\!&
 \ \ \ \ \ \ \ \ \ \ \ \ \ \ \
\tau(L_{m,i})\phi\big(g(x)\big)=\frac{1}{2}L_{2m,i}\frac{\lambda^{\frac{1}{2}}}{\sqrt{2}}tg\Big(\frac{t^{2}}{2}\Big)
=\frac{\lambda^{m+\frac{1}{2}}}{\sqrt{2}}t\Big(\delta_{i,0}\Big(\frac{t^{2}}{2}-mqa-\frac{mq}{2}\Big)
+d_{i}\Big)g\Big(\frac{t^{2}}{2}-mq\Big),\\
 &\!\!\!\!\!\!\!\!\!\!\!\!\!\!\!\!\!\!\!\!\!\!\!\!&
 \ \ \ \ \ \ \ \ \ \ \ \ \ \ \
\tau(G_{r,j})\phi\big(f(t)\big)\ = \frac{1}{\sqrt{2}}G_{2r,j}f\Big(\frac{t^{2}}{2}\Big)
=\frac{\lambda^{r}}{\sqrt{2}}\delta_{j,0}tf\Big(\frac{t^{2}}{2}-rq\Big), \\
&\!\!\!\!\!\!\!\!\!\!\!\!\!\!\!\!\!\!\!\!\!\!\!\!&
 \ \ \ \ \ \ \ \ \ \ \ \ \ \ \
\tau(G_{r,j})\phi\big(g(x)\big)\ = \frac{1}{\sqrt{2}}G_{2r,j}\frac{\lambda^{\frac{1}{2}}}{\sqrt{2}}tg\Big(\frac{t^{2}}{2}\Big)
=q\lambda^{r+\frac{1}{2}}\Big(\delta_{j,0}\Big(\frac{t^{2}}{2}-2rqa\Big)+2d_{j}\Big)g\Big(\frac{t^{2}}{2}-rq\Big).
\end{eqnarray*}
Hence, we are able to deduce that
\begin{eqnarray*}&\!\!\!\!\!\!\!\!\!\!\!\!\!\!\!\!\!\!\!\!\!\!\!\!&
\ \ \ \ \ \ \ \ \ \ \ \ \ \phi\big(L_{m,i}f(t)\big)=\tau(L_{m,i})\phi\big(f(t)\big),\ \ \ \ \ \ \ \ \phi\big(L_{m,i}g(x)\big)=\tau(L_{m,i})\phi\big(g(x)\big),  \\
  &\!\!\!\!\!\!\!\!\!\!\!\!\!\!\!\!\!\!\!\!\!\!\!\!&
\ \ \ \ \ \ \ \ \ \ \ \ \ \phi\big(G_{r,j}f(t)\big)\,=\tau(G_{r,j})\phi\big(f(t)\big),\ \ \ \ \ \ \ \ \,\phi\big(G_{r,j}g(x)\big)\,=\tau(G_{r,j})\phi\big(g(x)\big).
\end{eqnarray*}
This implies that $\phi$ is an $\L$-module homomorphism, and thus  an $\L$-module isomorphism.
\end{proof}
Now we are ready to prove Theorem $\ref{theo1.6}$ (2) and (3), which give the sufficient and necessary conditions for $\L$-modules $\Omega_{\L}(\lambda,a,b)$ to be simple, and the isomorphisms between them respectively.\vskip6pt

\noindent{\it Proof of Theorem $\ref{theo1.6}$} (2) and (3).
Let $\lambda,\,\mu\in \C^{*}$ and $a,\,a',\,b,\,b'\in \C$. 
Theorem $\ref{theo1.6}$ (2) follows from the above Propositions $\ref{prop-3.2a}$, $\ref{prop-3.3a}$ and Theorem $\ref{theo1.5}$ (2).
By Proposition $\ref{prop-3.3a}$, we have $\Omega_{\L}(\lambda,a,b)\cong \Omega_{\R}(\lambda^{\frac{1}{2}},a,2b)$ and $\Omega_{\L}(\mu,a',b')\cong \Omega_{\R}(\mu^{\frac{1}{2}},a',2b')$. Thus, from Theorem $\ref{theo1.5}$ (3), we get $\Omega_{\L}(\lambda,a,b)\cong \Omega_{\L}(\mu,a',b')$ if and only if $\lambda=\mu$, $a=a'$ and $b=b'$. 
\hfill$\Box$

\section{Classification of free $U(\mathfrak{h})$-modules of rank 2 over $\L$}
Let $\mathfrak{h}=\C L_{0,0}$. 
Then, $U(\mathfrak{h})=\C [L_{0,0}]$. Using the fact in \eqref{J_a_b_c d} that $\mathcal{L}$ is generated by $\{ G_{\frac{1}{2}+r,i}\,|\,r\in \Z,\,i\in \Z_{+}\}$, by Theorem $\ref{lemma-3.4a}$, we see that there do not exist $\mathcal{L}$-modules which are free of rank 1 as $U(\mathfrak{h})$-modules.

Let $V=V_{\overline{0}}\oplus V_{\overline{1}}$ be an $\mathcal{L}$-module such that it is free of rank 2 as a $\C [L_{0,0}]$-module with two homogeneous basis elements $v$ and $v'$. Then $v$ and $v'$ have different parities. Set
$$v=1_{\overline{0}}\in V_{\overline{0}},\ \ \ \ \
v'=1_{\overline{1}}\in V_{\overline{1}},\ \ \ \ \ L_{0,0}1_{\overline{0}}=t1_{\overline{0}}, \ \ \ \ \
  L_{0,0}1_{\overline{1}}=x1_{\overline{1}}.$$ We may assume that $V=\C [t]1_{\overline{0}}\oplus \C [x]1_{\overline{1}}$ with $V_{\overline{0}}=\C [t]1_{\overline{0}}$ and $V_{\overline{1}}=\C [x]1_{\overline{1}}$.

Recall that the Witt algebra $\W$ defined in \eqref{Witt-1} and the Heisenberg-Virasoro algebra $\H$ (modulo some center) defined in \eqref{Hei-1}--\eqref{Hei-3} can be regarded as subalgebras of $\L_{\overline{0}}$ for $q\neq0$ and $q=-1$ respectively. Then, we can regard both $V_{\overline{0}}$ and $V_{\overline{1}}$ as $\W$-modules and $\H$-modules for $q\neq-1,0$ and $q=-1$ respectively. Hence, there exist $\lambda,\,\mu\in \C^{*}$ and $a,\,b,\,c,\,d\in \C$ such that
\begin{eqnarray}&\!\!\!\!\!\!\!\!\!\!\!\!\!\!\!\!\!\!\!\!\!\!\!\!&\label{eqi1}
  L_{m,i}f(t)1_{\overline{0}}=\lambda^{m}\big(\delta_{i,0}\big(t-mqa\big)+\delta_{q,-1}\delta_{i,1}b\big)f(t-mq)1_{\overline{0}}, \\
  &\!\!\!\!\!\!\!\!\!\!\!\!\!\!\!\!\!\!\!\!\!\!\!\!&\label{eqi2}
L_{m,i}g(x)1_{\overline{1}}=\mu^{m}\big(\delta_{i,0}\big(x-mqc\big)+\delta_{q,-1}\delta_{i,1}d\big)g(x-mq)1_{\overline{1}}
\end{eqnarray}
for all $f(t)\in \C [t]$, $g(x)\in \C [x]$, $m\in \Z$ and $i\in \Z_{+}$. We further need the following preliminary results for later use.
\begin{lemm}\label{lemma-3.4ab}Keep the notations as above. Then $\lambda=\mu$, $\delta_{q,-1}b=\delta_{q,-1}d$ and there exists $e\in \C^{*}$ such that one of the following two cases occurs.
\rm\begin{itemize}\parskip-3pt
  \item [(i)]$c=a+\frac{1}{2},\ \ G_{\frac{1}{2},0}1_{\overline{0}}=e1_{\overline{1}},\ \ G_{\frac{1}{2},1}1_{\overline{0}}=0,\  \ G_{\frac{1}{2},1}1_{\overline{1}}=\frac{2q}{e}\lambda\delta_{q,-1}b1_{\overline{0}},\ \ G_{\frac{1}{2},0}1_{\overline{1}}=\frac{q}{e}\lambda(t-qa)1_{\overline{0}}$.
\item [(ii)]$c=a-\frac{1}{2},\ \ G_{\frac{1}{2},0}1_{\overline{1}}=e1_{\overline{0}},\ \ G_{\frac{1}{2},1}1_{\overline{1}}=0,\ \ G_{\frac{1}{2},1}1_{\overline{0}}=\frac{2q}{e}\lambda\delta_{q,-1}b1_{\overline{1}},\ \  G_{\frac{1}{2},0}1_{\overline{0}}=\frac{q}{e}\lambda(x-qa+\frac{q}{2})1_{\overline{1}}$.
\end{itemize}

\end{lemm}
\begin{proof}\vspace*{-5pt}Assume $$G_{\frac{1}{2},0}1_{\overline{0}}=f(x)1_{\overline{1}},\ \ \ \ G_{\frac{1}{2},1}1_{\overline{0}}=g(x)1_{\overline{1}}, \ \ \ \ G_{\frac{1}{2},0}1_{\overline{1}}=h(t)1_{\overline{0}},\ \ \ \ G_{\frac{1}{2},1}1_{\overline{1}}=\widetilde{h}(t)1_{\overline{0}},$$ where $f(x),\,g(x)\in \C[x]$ and $h(t), \,\widetilde{h}(t)\in \C[t]$. Then, we are able to have the \vspace*{-5pt}following
\begin{eqnarray*}\nonumber
\!\!\!\!\!\!\!\!\!\!\!\!\!\!\!\!\!\!\!\!\!\!\!\!\!&\!\!\!\!\!\!\!\!\!\!\!\!\!\!\!&
\ \ \ \ \ \ \ \ \ \ \ \ \ \ \ \ G^{2}_{\frac{1}{2},0}1_{\overline{0}}=G_{\frac{1}{2},0}f(x)1_{\overline{1}}=f\Big(t-\frac{q}{2}\Big)h(t)1_{\overline{0}},\\
\nonumber
\!\!\!\!\!\!\!\!\!\!\!\!\!\!\!\!\!\!\!\!\!\!\!\!\!&\!\!\!\!\!\!\!\!\!\!\!\!\!\!\!&
\ \ \ \ \ \ \ \ \ \ \ \ \ \ \ \ G^{2}_{\frac{1}{2},0}1_{\overline{1}}=G_{\frac{1}{2},0}h(t)1_{\overline{0}}=h\Big(x-\frac{q}{2}\Big)f(x)1_{\overline{1}}.
\end{eqnarray*}
Hence, according to $[G_{\frac{1}{2},0},G_{\frac{1}{2},0}]=2qL_{1,0}$, together with $(\ref{eqi1})$ and $(\ref{eqi2})$,
we get $f\big(t-\frac{q}{2}\big)h(t)=q\lambda(t-qa)$ and $h\big(x-\frac{q}{2}\big)f(x)=q\mu(x-qc)$, which imply that $\lambda=\mu$ and there exists $e\in \C^{*}$ such \vspace*{-5pt}that
\begin{eqnarray}&\!\!\!\!\!\!\!\!\!\!\!\!\!\!\!\!\!\!\!\!\!\!\!\!&\label{eqi3}
 \ \ \ \ \ \ c=a+\frac{1}{2},\ \ f(t)=e\ \ \mbox{and}\ \ h(t)=\frac{q}{e}\lambda\big(t-qa\big)\ \ \ \mbox{or} \\
  &\!\!\!\!\!\!\!\!\!\!\!\!\!\!\!\!\!\!\!\!\!\!\!\!&\label{eqi4}
 \ \ \ \ \ \ c=a-\frac{1}{2},\ \ h(t)=e\ \ \mbox{and}\ \ f(t)= \frac{q}{e}\lambda\Big(t-qa+\frac{q}{2}\Big).
\end{eqnarray}
Furthermore, from $[G_{\frac{1}{2},0},G_{\frac{1}{2},1}]1_{\overline{0}}=2qL_{1,1}1_{\overline{0}}=2q\lambda\delta_{q,-1}b1_{\overline{0}}$, we have
\begin{eqnarray*}\!\!\!\!\!\!\!\!\!\!\!\!&\!\!\!\!\!\!\!\!\!\!\!\!\!\!\!&
\ \ \ \ \ \ \ \ \ \ \ G_{\frac{1}{2},0}G_{\frac{1}{2},1}1_{\overline{0}}+G_{\frac{1}{2},1}G_{\frac{1}{2},0}1_{\overline{0}}=G_{\frac{1}{2},0}g(x)1_{\overline{1}}+G_{\frac{1}{2},1}f(x)1_{\overline{1}}
\nonumber\\\!\!\!\!\!\!\!\!\!\!\!\!&\!\!\!\!\!\!\!\!\!\!\!\!\!\!\!&
\ \ \ \ \ \ \ \ \ \ \ \ \ \ \ \ \ \ \ \ \ \ \ \ \ \ \ \ \ \ \ \ \ \ \ \ \ \ \ \ \ \ \ \
=g\Big(L_{0,0}-\frac{q}{2}\Big)G_{\frac{1}{2},0}1_{\overline{1}}+f\Big(L_{0,0}-\frac{q}{2}\Big)G_{\frac{1}{2},1}1_{\overline{1}}
\nonumber\\\!\!\!\!\!\!\!\!\!\!\!\!&\!\!\!\!\!\!\!\!\!\!\!\!\!\!\!&
\ \ \ \ \ \ \ \ \ \ \ \ \ \ \ \ \ \ \ \ \ \ \ \ \ \ \ \ \ \ \ \ \ \ \ \ \ \ \ \ \ \ \ \
=g\Big(t-\frac{q}{2}\Big)h(t)1_{\overline{0}}+f\Big(t-\frac{q}{2}\Big)\widetilde{h}(t)1_{\overline{0}},
\end{eqnarray*}
which \vspace*{-5pt}implies 
\begin{equation}\label{j16}
  g\big(t-\frac{q}{2}\big)h(t)+f\big(t-\frac{q}{2}\big)\widetilde{h}(t)=2q\lambda\delta_{q,-1}b.
\end{equation}
Similarly, by $[G_{\frac{1}{2},0},G_{\frac{1}{2},1}]1_{\overline{1}}=2qL_{1,1}1_{\overline{1}}=2q\mu\delta_{q,-1}d1_{\overline{1}}$, we get
\begin{equation}\label{j17}
  \widetilde{h}\Big(x-\frac{q}{2}\Big)f(x)+h\Big(x-\frac{q}{2}\Big)g(x)=2q\mu\delta_{q,-1}d.
\end{equation}
Note that $[G_{\frac{1}{2},1},G_{\frac{1}{2},1}]1_{\overline{0}}=2qL_{1,2}1_{\overline{0}}=0$, we can then obtain the following
\begin{equation}\label{eqj1}
  g\Big(t-\frac{q}{2}\Big)\widetilde{h}(t)=0.
\end{equation}
Taking $ f(t)=e$ and $h(t)=\frac{q}{e}\lambda\big(t-qa\big)$ in $(\ref{j16})$ and $(\ref{j17})$, and combining with $(\ref{eqj1})$, we obtain
\begin{equation}\label{j18}
  g(x)=0,\ \ \ \widetilde{h}(x)=\frac{2q}{e}\lambda\delta_{q,-1}b\ \ \ \mbox{and}\ \ \ \delta_{q,-1}b=\delta_{q,-1}d.
  \end{equation}
Taking $ h(t)=e$ and $f(t)=\frac{q}{e}\lambda\big(t-qa+\frac{q}{2}\big)$ in $(\ref{j16})$ and $(\ref{j17})$, and applying $(\ref{eqj1})$, we \vspace*{-5pt}have  
\begin{equation}\label{j19}
    \widetilde{h}(x)=0,\ \ \ g(x)=\frac{2q}{e}\lambda\delta_{q,-1}b\ \ \ \mbox{and}\ \ \ \delta_{q,-1}b=\delta_{q,-1}d.
\end{equation}
Therefore, Lemma \ref{lemma-3.4ab} follows from $(\ref{eqi3})$, $(\ref{eqi4})$, $(\ref{j18})$ and $(\ref{j19})$.
\end{proof}

According to Lemma $\ref{lemma-3.4ab}$, up to parity,  without loss of generality, we can assume the following\vspace*{-5pt}, $$c=a+\frac{1}{2}, \ \ G_{\frac{1}{2},0}1_{\overline{0}}=1_{\overline{1}},\ \ G_{\frac{1}{2},1}1_{\overline{0}}=0,\ \ G_{\frac{1}{2},0}1_{\overline{1}}=q\lambda(t-qa)1_{\overline{0}}\mbox{ \ \ and  \ \ } G_{\frac{1}{2},1}1_{\overline{1}}=2q\lambda\delta_{q,-1}b1_{\overline{0}}.$$
\begin{lemm}\label{lemma-3.4abc} For any $m\in \frac{1}{2}+\Z$ and $i\in \Z_{+}$, one \vspace*{-5pt}has
\begin{eqnarray}&\!\!\!\!\!\!\!\!\!\!\!\!\!\!\!\!\!\!\!\!\!\!\!\!&\label{ii1}
G_{m,i}1_{\overline{0}}=\lambda^{m-\frac{1}{2}}\delta_{i,0}1_{\overline{1}},\\
&\!\!\!\!\!\!\!\!\!\!\!\!\!\!\!\!\!\!\!\!\!\!\!\!&\label{ii2}
G_{m,i}1_{\overline{1}}=q\lambda^{m+\frac{1}{2}}\Big(\delta_{i,0}\big(t-2mqa\big)+2\delta_{q,-1}\delta_{i,1}b\Big)1_{\overline{0}}.
\end{eqnarray}
\end{lemm}
\begin{proof} At first, we prove that ($\ref{ii1}$) is true for any $m\in \frac{1}{2}+\Z$ and $i\in \Z_{+}$. By $(\ref{eqi1})$, $(\ref{eqi2})$ and Lemma $\ref{lemma-3.4ab}$, we get
\begin{eqnarray*}\!\!\!\!\!\!\!\!\!\!\!\!&\!\!\!\!\!\!\!\!\!\!\!\!\!\!\!&
\ \ \ \ \ \ \ \ \ \ \ \ \ \ \frac{1}{2}\Big(i-\Big(m-\frac{3}{2}\Big)q\Big)G_{m,i}1_{\overline{0}}=[L_{m-\frac{1}{2},i},G_{\frac{1}{2},0}]1_{\overline{0}}
=L_{m-\frac{1}{2},i}G_{\frac{1}{2},0}1_{\overline{0}}-G_{\frac{1}{2},0}L_{m-\frac{1}{2},i}1_{\overline{0}}
\nonumber\\\!\!\!\!\!\!\!\!\!\!\!\!&\!\!\!\!\!\!\!\!\!\!\!\!\!\!\!&
\ \ \ \ \ \ \ \ \ \ \ \ \ \ \ \ \ \ \ \ \ \ \ \ \ \ \ \ \ \ \ \ \ \ \ \ \ \ \ \ \ \ \ \ \ \ \
=L_{m-\frac{1}{2},i}1_{\overline{1}}-G_{\frac{1}{2},0}\lambda^{m-\frac{1}{2}}\Big(\delta_{i,0}\Big(t-mqa+\frac{qa}{2}\Big)
+\delta_{q,-1}\delta_{i,1}b\Big)1_{\overline{0}}
\nonumber\\\!\!\!\!\!\!\!\!\!\!\!\!&\!\!\!\!\!\!\!\!\!\!\!\!\!\!\!&
\ \ \ \ \ \ \ \ \ \ \ \ \ \ \ \ \ \ \ \ \ \ \ \ \ \ \ \ \ \ \ \ \ \ \ \ \ \ \ \ \ \ \ \ \ \ \
=\lambda^{m-\frac{1}{2}}\Big(\delta_{i,0}\Big(x-\Big(a+\frac{1}{2}\Big)\Big(m-\frac{1}{2}\Big)q\Big)+\delta_{q,-1}\delta_{i,1}b\Big)1_{\overline{1}}
\nonumber\\\!\!\!\!\!\!\!\!\!\!\!\!&\!\!\!\!\!\!\!\!\!\!\!\!\!\!\!&
\ \ \ \ \ \ \ \ \ \ \ \ \ \ \ \ \ \ \ \ \ \ \ \ \ \ \ \ \ \ \ \ \ \ \ \ \ \ \ \ \ \ \ \ \ \ \ \ \ \ \,
-\lambda^{m-\frac{1}{2}}\Big(\delta_{i,0}\Big(x-mqa+\frac{qa}{2}-\frac{q}{2}\Big)+\delta_{q,-1}\delta_{i,1}b\Big)1_{\overline{1}}
\nonumber\\\!\!\!\!\!\!\!\!\!\!\!\!&\!\!\!\!\!\!\!\!\!\!\!\!\!\!\!&
\ \ \ \ \ \ \ \ \ \ \ \ \ \ \ \ \ \ \ \ \ \ \ \ \ \ \ \ \ \ \ \ \ \ \ \ \ \ \ \ \ \ \ \ \ \ \
=\frac{1}{2}\Big(\frac{3}{2}-m\Big)q\lambda^{m-\frac{1}{2}}\delta_{i,0}1_{\overline{1}}
,\end{eqnarray*}
\vspace*{-5pt}and
\begin{eqnarray*}\!\!\!\!\!\!\!\!\!\!\!\!&\!\!\!\!\!\!\!\!\!\!\!\!\!\!\!&
  \ \ \ \ \ \ \frac{1}{2}\Big(i-2m-\Big(m-\frac{3}{2}\Big)q\Big)G_{m,i}1_{\overline{0}}=[L_{m-\frac{1}{2},i-1},G_{\frac{1}{2},1}]1_{\overline{0}}
\nonumber\\\!\!\!\!\!\!\!\!\!\!\!\!&\!\!\!\!\!\!\!\!\!\!\!\!\!\!\!&
\ \ \ \ \ \ \ \ \ \ \ \ \ \ \ \ \ \ \ \ \ \ \ \ \ \ \ \ \ \ \ \ \ \ \ \ \ \ \ \ \ \ \ \ \ \ \ \,
=L_{m-\frac{1}{2},i-1}G_{\frac{1}{2},1}1_{\overline{0}}-G_{\frac{1}{2},1}L_{m-\frac{1}{2},i-1}1_{\overline{0}}
=0,
\end{eqnarray*}
which imply that
\begin{eqnarray}&\!\!\!\!\!\!\!\!\!\!\!\!\!\!\!\!\!\!\!\!\!\!\!\!&\label{eqi123}
\Big(i-\Big(m-\frac{3}{2}\Big)q\Big)G_{m,i}1_{\overline{0}}=\Big(\frac{3}{2}-m\Big)q\lambda^{m-\frac{1}{2}}\delta_{i,0}1_{\overline{1}},\\
&\!\!\!\!\!\!\!\!\!\!\!\!\!\!\!\!\!\!\!\!\!\!\!\!&\label{eqi12345}
\Big(i-2m-\Big(m-\frac{3}{2}\Big)q\Big)G_{m,i}1_{\overline{0}}=0.
\end{eqnarray}
For any $m\in \frac{1}{2}+\Z$ and $m\neq\frac{3}{2}$, by $(\ref{eqi123})$ and $(\ref{eqi12345})$, it follows that
\begin{equation}\label{eqi1234}
  G_{m,i}1_{\overline{0}}=\lambda^{m-\frac{1}{2}}\delta_{i,0}1_{\overline{1}}.
\end{equation}
Take $m=\frac{3}{2}$ in $(\ref{eqi123})$, we get $G_{\frac{3}{2},i}1_{\overline{0}}=0$ for $i\in \N$.
Furthermore, from $(\ref{eqi1})$, $(\ref{eqi2})$ and $(\ref{eqi1234})$, we can obtain the \vspace*{-5pt}following
\begin{eqnarray*}\!\!\!\!\!\!\!\!\!\!\!\!&\!\!\!\!\!\!\!\!\!\!\!\!\!\!\!&
\ \ \ \ \ \ \ \ \ \ \ \ \ \ -\frac{3q}{2}G_{\frac{3}{2},0}1_{\overline{0}}=[L_{2,0},G_{-\frac{1}{2},0}]1_{\overline{0}}
\nonumber\\\!\!\!\!\!\!\!\!\!\!\!\!&\!\!\!\!\!\!\!\!\!\!\!\!\!\!\!&
\ \ \ \ \ \ \ \ \ \ \ \ \ \ \ \ \ \ \ \ \ \ \ \ \ \ \ \ \
=L_{2,0}G_{-\frac{1}{2},0}1_{\overline{0}}-G_{-\frac{1}{2},0}L_{2,0}1_{\overline{0}}
\nonumber\\\!\!\!\!\!\!\!\!\!\!\!\!&\!\!\!\!\!\!\!\!\!\!\!\!\!\!\!&
\ \ \ \ \ \ \ \ \ \ \ \ \ \ \ \ \ \ \ \ \ \ \ \ \ \ \ \ \
=\lambda^{-1}L_{2,0}1_{\overline{1}}-\lambda^{2}G_{-\frac{1}{2},0}\big(t-2qa\big)1_{\overline{0}}
\nonumber\\\!\!\!\!\!\!\!\!\!\!\!\!&\!\!\!\!\!\!\!\!\!\!\!\!\!\!\!&
\ \ \ \ \ \ \ \ \ \ \ \ \ \ \ \ \ \ \ \ \ \ \ \ \ \ \ \ \
=\lambda\Big(x-\big(2a+1\big)q\Big)1_{\overline{1}}-\lambda\Big(x-2qa+\frac{q}{2}\Big)1_{\overline{1}}
=-\frac{3q}{2}\lambda1_{\overline{1}},
\end{eqnarray*}
Hence, $G_{\frac{3}{2},0}1_{\overline{0}}=\lambda1_{\overline{1}}$, and $(\ref{ii1})$ holds.

Now we want to show that $(\ref{ii2})$ holds. Firstly we consider the case $i=0$. Suppose  $G_{m,0}1_{\overline{1}}=\widetilde{f}(t)1_{\overline{0}}$, where $\widetilde{f}(t)\in \C[t]$. Then, applying $(\ref{eqi2})$ and $(\ref{ii1})$, we \vspace*{-5pt}obtain
\begin{equation*}
  G_{m,0}^{2}1_{\overline{1}}=qL_{2m,0}1_{\overline{1}}=q\lambda^{2m}\big(x-2mqa-mq\big)1_{\overline{1}}
,\end{equation*}
\vspace*{-5pt}and
\begin{equation*}
  G_{m,0}G_{m,0}1_{\overline{1}}=G_{m,0}\widetilde{f}(t)1_{\overline{0}}=\widetilde{f}(x-mq)G_{m,0}1_{\overline{0}}
=\lambda^{m-\frac{1}{2}}\widetilde{f}(x-mq)1_{\overline{1}},
\end{equation*}
which yield that $\lambda^{m-\frac{1}{2}}\widetilde{f}(x-mq)=q\lambda^{2m}\big(x-2mqa-mq\big)$. Consequently, we \vspace*{-5pt}have
\begin{equation}\label{eqi567}
  G_{m,0}1_{\overline{1}}=q\lambda^{m+\frac{1}{2}}\big(t-2mqa\big)1_{\overline{0}}.
\end{equation}
Next we consider the case $i\in \N$. 
By Lemma $\ref{lemma-3.4ab}$ and a similar proof of $(\ref{eqi123})$ and $(\ref{eqi12345})$, we \vspace*{-5pt}obtain
\begin{eqnarray}&\!\!\!\!\!\!\!\!\!\!\!\!\!\!\!\!\!\!\!\!\!\!\!\!&\label{eqi678}
\frac{1}{2} \Big(i-\Big(m-\frac{3}{2}\Big)q\Big)G_{m,i}1_{\overline{1}}=\Big(\frac{1}{2}-m\Big)\lambda^{m+\frac{1}{2}}\delta_{q,-1}\delta_{i,1}b1_{\overline{0}},\\[4pt]
&\!\!\!\!\!\!\!\!\!\!\!\!\!\!\!\!\!\!\!\!\!\!\!\!&\label{eqi789}
\frac{1}{2}\Big(i-2m-\Big(m-\frac{3}{2}\Big)q\Big)G_{m,i}1_{\overline{1}}=\Big(\frac{1}{2}+m\Big)\lambda^{m+\frac{1}{2}}\delta_{q,-1}\delta_{i,1}b1_{\overline{0}}.
\end{eqnarray}
Let $m\neq\frac{1}{2}$ in $(\ref{eqi678})$ and $(\ref{eqi789})$, it follows that $G_{m,i}1_{\overline{1}}=2q\lambda^{m+\frac{1}{2}}\delta_{q,-1}\delta_{i,1}b1_{\overline{0}}$ for $i\in \N$. Further, let $m=\frac{1}{2}$ in $(\ref{eqi678})$ and $(\ref{eqi789})$, it follows that $G_{\frac{1}{2},i}1_{\overline{1}}=2q\lambda\delta_{q,-1}\delta_{i,1}b1_{\overline{0}}$ for $i\in \N$. Hence, combining with $(\ref{eqi567})$, we get $(\ref{ii2})$. This completes the proof of Lemma \ref{lemma-3.4abc}.
\end{proof}

Finally, we can give the proof of Theorem $\ref{theo1.6}$ (4), which gives a complete classification of free $U(\mathfrak{h})$-modules of rank 2 over the Neveu-Schwarz-Block algebra $\L$.

Since $G_{r,i}f(t)1_{\overline{0}}=f(x-rq)G_{r,i}1_{\overline{0}}$ and $G_{r,i}g(x)1_{\overline{1}}=g(t-rq)G_{r,i}1_{\overline{1}}$, Theorem $\ref{theo1.6}$ (4) is implied by $(\ref{eqi1})$, $(\ref{eqi2})$ and Lemma $\ref{lemma-3.4abc}$\vspace*{-5pt}. 

\end{CJK*}

\begin{thebibliography}{9999}\vskip0pt\small
\parindent=2ex\parskip=-1.5pt\baselineskip=-1.5pt\lineskip=3.0pt
\def\re#1{\bibitem{#1}\label{#1}}

\re{B1} I. Bakas, The structure of the $\W_{\infty}$ algebra, {\em Comm. Math. Phys.} {\bf134} (1990), 487--508.
\re{B} R. E. Block, On torsion-free abelian groups and Lie algebras, {\em Proc. Amer. Math. Soc.} {\bf9} (1958), 613--620.

\re{CG} H. Chen, X. Guo, Non-weight modules over the Heisenberg-Virasoro algebra and the $W$ algebra $W(2, 2)$, {\em J. Algebra Appl.} {\bf16} (2017), 1750097.
\re{CG1} H. Chen, X. Guo, A new family of modules over the Virasoro algebra, {\em J. Algebra} {\bf457} (2016), 73--105.
\re{CY} Q. Chen, Y. Yao, Non-weight modules over algebras related to the Virasoro algebra, {\em J. Geom. Phys.} {\bf134} (2018), 11--18.
\re{DZ} D. Z. Dokovic, K. Zhao, Derivations, isomorphisms and second cohomology of generalized Block algebras,  {\em Algebra Colloq.} {\bf3} (1996), 245--272.

\re{HCS} J. Han, Q. Chen, Y. Su, Modules over the algebra $\mathcal{V}ir(a, b)$, {\em Linear Algebra Appl.} {\bf 515} (2017), 11--23.
\re{KPS} E. Kirkman, C. Procesi, L. Small, A $q$-analog for the Virasoro algebra, {\em Comm. Algebra} {\bf22} (1994), 3755--3774.

\re{18}	C. Li, J. He, Y. Su, Block (or Hamiltonian) Lie symmetry of dispersionless D-type Drinfeld-Sokolov hierarchy, {\em Commun. Theor. Phys.} {\bf 61} (2014), 431--435.
\re{27}	C. Li, J. He, Y. Su, Block type symmetry of bigraded Toda hierarchy, {\em J. Math. Phys.} {\bf53} (2012), 013517.

\re{LG} X. Liu, X. Guo, $U(\mathfrak{h})$-free modules over the Block algebra $\mathcal{B}(q)$, arXiv:1801.03232v1.
\re{LZ} R. Lu, K. Zhao, Irreducible Virasoro modules from irreducible Weyl modules, {\em J. Algebra} {\bf414} (2014), 271--287.
\re{N} J. Nilsson, Simple $sl_{n+1}$-module structures on $U(\mathfrak{h})$, {\em J. Algebra} {\bf424} (2015), 294--329.
\re{N1} J. Nilsson, $U(\mathfrak{h})$-free modules and coherent families, {\em J. Pure Appl. Algebra} {\bf220(4)} (2016), 1475--1488.
\re{OZ} J. M. Osborn, K. Zhao, Infinite-dimensional Lie algebras of generalized Block type, {\em Proc. Amer. Math. Soc.} {\bf127} (1999), 1641--1650.
\re{SZ} Y. Su, K. Zhao, Generalized Virasoro and super-Virasoro algebras and modules of the intermediate series, {\em J. Algebra} {\bf252} (2002), 1--19.

\re{TZ} H. Tan, K. Zhao, $\W_{n}^{+}$ and $\W_{n}$-module structures on $U(\mathfrak{h}_{n})$, {\em J. Algebra}  {\bf424} (2015), 357--375.
\re{TZ1} H. Tan, K. Zhao, Irreducible modules over Witt algebras $\W_{n}$ and over $sl_{n+1}(\C)$, {\em Algebr. Represent. Theory} {\bf21(4)} (2018), 787--806.
\re{WZ} Y. Wang, H. Zhang, A class of non-weight modules over the Schr$\ddot{o}$dinger-Virasoro algebras, arXiv:1809.05236v1.
\re{X1} C. Xia, Structure of two classes of Lie superalgebras of Block type, {\em Int. J. Math.} {\bf 27(5)} (2016), 1650038.

\re{X} X. Xu, New generalized simple Lie algebras of Cartan type over a field with characteristic 0, {\em J. Algebra} {\bf224} (2000), 23--58.
\re{YYX} H. Yang, Y. Yao, L. Xia, On non-weight representations of the $N = 2$ superconformal algebras, {\em J. Pure Appl. Algebra} {\bf225} (2021), 106529.
\re{YYX1} H. Yang, Y. Yao, L. Xia, A family of non-weight modules over the super-Virasoro algebras, {\em J. Algebra} {\bf547} (2020), 538--555.
\re{ZZ} Q. Zhang, Y. Zhang, Derivation algebras of the modular Lie superalgebras $W$ and $S$ of Cartan-type, {\em Acta Math. Sci.} {\bf 20} (2000), 137--144.
\re{Z} K. Zhao, A class of infinite dimensional simple Lie algebras, {\em J. London Math. Soc.} {\bf62} (2000), 71--84.
\end{thebibliography}
 \end{document}